\documentclass[12pt]{article}

\usepackage{amssymb,url}
\usepackage{amsfonts}
\usepackage{theorem}
\usepackage[leqno]{amsmath}
\usepackage{enumerate}
\usepackage{indentfirst, color}

\newenvironment{proof}[1][Proof]{\textbf{#1.} }{\hfill $\square$}

\newtheorem{lem}{Lemma}
\newtheorem{thm}{Theorem}
\newtheorem{defin}{Definition}
\newtheorem{rem}{Remark}
\newtheorem{prop}{Proposition}

{\begin{flushleft}\textbf{#1} \textbf{#2}} %
{\end{flushleft}}

\newcommand{\bord}{\partial \mathcal{S}}
\newcommand{\tOm}{\widetilde{\Omega}}
\newcommand{\R}{\mathbb{R}}
\newcommand{\N}{\mathbb{N}}
\newcommand{\Prb}{\mathbb{P}}
\newcommand{\E}{\mathbb{E}}
\newcommand{\bD}{\mathbb{D}}

\newcommand{\bH}{\mathbb{H}}
\newcommand{\bL}{\mathbb{L}}
\newcommand{\bS}{\mathbb{S}}
\newcommand{\bF}{\mathbb{F}}
\newcommand{\tri}{\mathcal{F}}
\newcommand{\cP}{\mathcal{P}}
\newcommand{\cR}{\mathcal{R}}
\newcommand{\cH}{\mathcal{H}}
\newcommand{\cB}{\mathcal{B}}

\newcommand{\cD}{\mathcal{D}}

\newcommand{\cS}{\mathcal{S}}
\newcommand{\cE}{\mathcal{E}}
\newcommand{\cA}{\mathcal{A}}
\newcommand{\cI}{\mathcal{I}}

\newcommand{\tP}{\widetilde{\mathcal{P}}}
\newcommand{\eps}{\varepsilon}

\newcommand{\al}{\alpha}

\newcommand{\la}{\lambda}
\newcommand{\ind}{\mathbf{1}}

\newcommand{\tr}{\text{Trace}}

\newcommand{\tmu}{\widetilde{\mu}}

\newcommand{\ope}{\mathcal{L}}

\newcommand{\cad}{c\`adl\`ag }

\title{Integro-partial differential equations with singular terminal condition.}
\author{Alexandre Popier \\*
LUNAM Universit\'e, Universit\'e du Maine, \\* Laboratoire Manceau de Math\'ematiques,\\*
 Avenue O. Messiaen, 72085 Le Mans cedex 9\\* France}

\begin{document}

\maketitle

\begin{abstract}
In this paper, we show that the minimal solution of a backward stochastic differential equation gives a probabilistic representation of the minimal viscosity solution of an integro-partial differential equation both with a singular terminal condition. Singularity means that at the final time, the value of the solution can be equal to infinity. Different types of regularity of this viscosity solution are investigated: Sobolev, H\"older or strong regularity. 
\end{abstract}

\vspace{0.7cm}


\section*{Introduction}

The notion of backward stochastic differential equations (BSDEs) was first introduced by Bismut in \cite{bism:73} in the linear setting and by Pardoux \& Peng in \cite{pard:peng:90} for non linear equation. One particular interest for the study of BSDE is the application to partial differential equations (PDEs). Indeed as proved by Pardoux \& Peng in \cite{pard:peng:92}, BSDEs can be seen as generalization of the Feynman-Kac formula for non linear PDEs. Roughly speaking, if we can solve a system of two SDEs with one forward in time and one backward in time, then the solution is a deterministic function and is a (weak) solution of the related PDE. This is a method of characteristics to solve the parabolic PDE. 
The converse assertion can be proved provided we can apply It\^o's formula, that is if the solution of the PDE is regular enough. Since then a large literature has been developed on this topic (see in particular the books \cite{delo:13}, \cite{book:bsde},  \cite{pard:rasc:14} and the references therein). The extension to quasi-linear PDEs or to fully non linear PDEs has been already developed (see among other \cite{ma:yong:zhao:10} and \cite{touz:10}). 

Here we are interesting in another development of the theory: the case of integro-partial differential equation (IPDE) and of (backward) SDEs with Poisson random noise. In \cite{barl:buck:pard:97}, Barles, Buckdahn \& Pardoux show that we can add in the system of forward backward SDE a Poisson random measure and if we can find a solution to this system, again the solution is a weak solution of an IPDE:
\begin{equation}\label{eq:IPDE}
\frac{\partial }{\partial t} u(t,x) + \ope u(t,x) + \mathcal{I}(t,x,u) + f(t,x,u,(\nabla u)\sigma,\mathcal{B}(t,x,u)) = 0
\end{equation}
with terminal condition $u(T,.) = g$. Here $\ope$ is a local second-order differential operator corresponding to the infinitesimal generator of the continuous part of the forward SDE and $\mathcal{I}$ and $\mathcal{B}$ are two integro-differential operators. $\cI$ is the discontinuous part of the infinitesimal generator of the forward SDE, and $\cB$ is related to the generator of the BSDE. In \cite{barl:buck:pard:97}, weak solution means viscosity solution. Since this paper, several authors have weaken the assumptions of \cite{barl:buck:pard:97}. The book \cite{delo:13} (Chapter 4) gives a nice review of these results (and several references on this topic). 

\vspace{0.5cm}
Among all semi-linear PDEs, a particular form has been widely studied:
\begin{equation}\label{eq:first_sing_PDE}
\frac{\partial u}{\partial t} (t,x) + \ope u(t,x) -u(t,x)|u(t,x)|^q = 0. 
\end{equation}
Baras \& Pierre \cite{bara:pier:84}, Marcus \& Veron \cite{marc:vero:99} (and many other papers) have given existence and uniqueness results for this PDE. In \cite{marc:vero:99} it is shown that every positive solution of \eqref{eq:first_sing_PDE} possesses a uniquely determined final trace $g$ which can be represented by a couple $(\mathcal{S},\mu)$ where $\mathcal{S}$ is a closed subset of $\R^{d}$ and $\mu$ a non negative Radon measure on $\mathcal{R} = \R^{d} \setminus \mathcal{S}$. The final trace can also be represented by a positive, outer regular Borel measure $\nu$, and $\nu$ is not necessary locally bounded. The two representations are related by:
$$\forall A \subset \R^{m}, A \ \mbox{Borel}, \ \left\{ \begin{array}{ll}
\nu (A) = \infty & \ \mbox{if} \ A \cap \mathcal{S} \neq \emptyset \\
\nu (A) = \mu (A) & \ \mbox{if} \ A \subset \mathcal{R}.
\end{array} \right.$$
The set $\mathcal{S}$ is the set of singular final points of $u$ and it corresponds to a ``blow-up'' set of $u$. From the probabilistic point of view Dynkin \& Kuznetsov \cite{dynk:kuzn:97} and Le Gall \cite{lega:96} have proved similar results for the PDE \eqref{eq:first_sing_PDE} in the case $0< q \leq 1$ using the theory of superprocesses. Now if we want to represent the solution $u$ of \eqref{eq:first_sing_PDE} using a FBSDE, we have to deal with a {\it singular} terminal condition $\xi$ in the BSDE, which means that $\Prb(\xi = +\infty)> 0$. This singular case and the link between the solution of the BSDE with singular terminal condition and the viscosity solution of the PDE \eqref{eq:first_sing_PDE} have been studied first in \cite{popi:06}. Recently it was used to solve a stochastic control problem for portfolio liquidation (see \cite{anki:jean:krus:13} or \cite{grae:hors:qiu:13}). In \cite{krus:popi:15} we enlarge the known results on this subject for more general generator $f$ (than $f(y)=-y|y|^q$). 

\vspace{0.5cm}
In this paper our goal is to generalize the results of \cite{popi:06} and using our recent papers \cite{krus:popi:15} and \cite{popi:16} we want to study the related IPDE \eqref{eq:IPDE} when the terminal condition $u(T,.) = g$ is {\it singular} in the sense that $g$ takes values in $\R_+ \cup \{+\infty\}$ and the set
$$\cS = \{x \in \R^d, \quad g(x) = +\infty \}$$
is a non empty closed subset of $\R^d$. Again in the non singular case, if the terminal function $g$ is of linear growth, the relation between the FBSDE and the IPDE is obtained in \cite{barl:buck:pard:97}. Moreover several papers have studied the existence and the uniqueness of the solution of such IPDE (see among others \cite{alva:tour:96}, \cite{barl:imbe:08}, \cite{bens:saya:02} or \cite{hama:zhao:14}). 
\begin{itemize}
\item To our best knowledges {\bf the study of \eqref{eq:IPDE} with a singularity at time $T$} is completely new. There is no probabilistic representation of such IPDE using superprocesses and no deterministic or analytic works on this topic. In the PhD thesis of Piozin \cite{pioz:15} (as in \cite{popi:06}), we have studied the case when $f(t,y,z,u)=f(y) = -y|y|^q$. Hence the aim of the paper is to prove that this minimal solution $Y$ of the singular BSDE is the probabilistic representation of the minimal positive viscosity solution $u$ of the IPDE for general function $f$ with a singular terminal condition. 
\item One {\it applied motivation} for this study (optimal liquidation of a portfolio) is developed in \cite{anki:jean:krus:13}, \cite{grae:hors:qiu:13} and \cite{krus:popi:15}. The optimal solution of a stochastic control problem with terminal constraint is the minimal solution $Y$ of the singular FBSDE. The value function $v$ (and the optimal state) can be computed directly with $Y$. In other words from this paper we obtain that $v$ is the minimal viscosity solution of \eqref{eq:IPDE} with singular terminal condition.
\end{itemize}
The ground of this paper has been already prepared by the works \cite{barl:buck:pard:97}, \cite{krus:popi:15}, \cite{popi:06} and \cite{popi:16}, especially for the most technical aspects. The novelty is that we gather the papers and we obtain non trivial conditions (for example between $\cI$ and $\cS$) for existence and minimality of the viscosity solution of \eqref{eq:IPDE} with singularity at time $T$.   

\vspace{0.5cm}
The paper is organized as follows. In the first part we describe the mathematical setting. Since we are interesting in singular terminal condition, the generator $f$ of the BSDE has to satisfy special conditions: when $y$ becomes large, the function $y\mapsto f(t,x,y,z,u)-f(t,x,y,0,0)$ decreases at least like $-y|y|^q$. We recall the precise result of \cite{barl:buck:pard:97} when the terminal condition is non singular. From the forward backward SDE we get a continuous viscosity solution of the equation \eqref{eq:IPDE}. We also give the result of \cite{krus:popi:15} concerning existence and minimality of a solution for a singular BSDE.

In the second section, we show that the minimal solution $Y$ of the singular BSDE provides the minimal viscosity solution $u$ for the IPDE with singular terminal condition. In details we show that $Y^{t,x}_t = u(t,x)$ is the minimal (discontinuous) viscosity solution of \eqref{eq:IPDE} on any interval $[0,T-\varepsilon]$ for $\varepsilon > 0$ with $\liminf_{t\to T} u(t,x) \geq g(x)$ (Theorem \ref{thm:visc_sol}). Here we mainly use an a priori estimate on the solution which gives an upper bound on the solution independent of the terminal condition. The structure of the generator $f$ is crucial here. Then we can apply a stability result on viscosity solutions: roughly speaking an increasing sequence of viscosity solutions is itself a viscosity solution. Minimality is obtained by a comparison result for viscosity solution for IPDE adapted for our setting. 

The singularity of the terminal condition becomes a main trouble if we want to prove that $\limsup_{t\to T} u(t,x) \leq g(x)$ on the regular set $\cR=\{g < +\infty\}$ (Theorem \ref{thm:sing_term_visc_sol}). As in \cite{popi:06}, we first prove that the solution is bounded (locally on $\cR$) by a localization argument. Then we derive that the upper semi-continuous envelop $u^*$ solves the IPDE \eqref{eq:IPDE} with a relaxed terminal condition. Finally we derive the wanted result. These steps can be done under extra assumptions between the jumps of the forward SDE (or the coefficients in the non local operator $\cI$) and the singular set $\cS= \{g = +\infty\}$.  

The last part is devoted to study the regularity of this minimal solution on $[0,T-\varepsilon]\times \R^d$. Indeed the minimal viscosity solution constructed before is the increasing limit of continuous functions. Hence it is lower semicontinuous, but the continuity is an open question. Therefore we give several conditions on the coefficients of the forward SDE and on the L\'evy measure $\lambda$ in order to obtain:
\begin{itemize}
\item Sobolev-type regularity: $u$ and $\nabla u$ are in some $\bL^2$ weighted space. Only the coefficients of the forward SDE are supposed to be regular. 
\item H\"older regularity of $u$. We will impose some conditions on $\lambda$, but no additional regularity condition on the parameters. 
\item Classical regularity: $u$ is of class $C^{1,2}$ on $[0,T)\times \R^d$. The matrix diffusion $\sigma$ is supposed to be uniformly elliptic and $\lambda$ is not too singular on $0$.
\end{itemize}
Our conditions are quite classical and widely used. Different sets of assumptions could be also used to obtain similar results and we do not claim that we are exhaustive. Let us emphasize that most existing results assume that the terminal condition for \eqref{eq:IPDE} is already smooth enough (Lipschitz continuous or $C^{2}$). In our case we only have boundedness far from the terminal time $T$ and we need to circumvent this difficulty.

\section{Setting and known results}

We consider a filtered probability space $(\Omega,\tri,\Prb,\bF = (\tri_t)_{t\geq 0})$. We assume that this set supports a $k$-dimensional Brownian motion $W$ and a Poisson random measure $\mu$ with intensity $\lambda(de)dt$ on the space $E \subset \R^{d'} \setminus \{0\}$. The filtration $\bF$ is generated by $W$ and $\mu$. We will denote $\cE$ the Borelian $\sigma$-field of $E$ and $\tmu$ is the compensated measure: for any $A \in \cE$ such that $\lambda(A)  < +\infty$, then $\tmu([0,t] \times A) = \mu([0,t] \times A) - t \lambda(A)$ is a martingale. The measure $\lambda$ is $\sigma$-finite on $(E,\cE)$ satisfying 
$$\int_E (1\wedge |e|^2) \lambda(de) < +\infty.$$
In this paper for a given $T\geq 0$, we denote by $\cP$ the predictable $\sigma$-field on $\Omega \times [0,T]$ and
$$\tP=\cP \otimes \cE.$$
On $\tOm = \Omega \times [0,T] \times E$, a function that is $\tP$-measurable, is called predictable. $G_{loc}(\mu)$ is the set of $\tP$-measurable functions $\psi$ on $\tOm$ such that for any $t \geq 0$ a.s.
$$ \int_0^t \int_E (|\psi_s(e)|^2\wedge |\psi_s(e)|) \lambda(de) < +\infty.$$
$\cD$ (resp. $\cD(0,T)$) is the set of all predictable processes on $\R_+$ (resp. on $[0,T]$). 
We refer to \cite{jaco:shir:03} for details on random measures and stochastic integrals. On $\R^d$, $|.|$ denotes the Euclidean norm whereas the symbol $\cdot$ stands for the inner product. The space $\R^{d\times d'}$ is identified with the space of real matrices with $d$ rows and $d'$ columns. If $z \in  \R^{d\times d'}$, we have $|z|^2 = \mbox{trace}(zz^*)$. 

Now to define the solution of our BSDE, let us introduce the following spaces for $p\geq 1$. $\bD^p(0,T)$ is the space of all adapted \cad processes $X$ such that $ \E \left(  \sup_{t\in [0,T]} |X_t|^p \right)$ is finite.
$\bH^p(0,T)$ denotes the subspace of all processes $X\in \cD(0,T)$ such that the expectation
$\E \left[ \left( \int_0^T |X_t|^2 dt\right)^{p/2} \right]$ is finite.
$\bL^p_\mu(0,T) = \bL^p_{\mu}(\Omega\times (0,T)\times E)$ is the set of processes $\psi \in G_{loc}(\mu)$ such that
$\E \left[ \left(  \int_0^T \int_{E} |\psi_s(e)|^2 \lambda(de) ds \right)^{p/2} \right] < +\infty$.
 $\bL^p_\lambda(E)=\bL^p(E,\lambda;\R^m)$ is the set of measurable functions $\psi : E \to \R^m$ with $\lambda$-integrable $p$ moment. Finally 
$$\bS^p(0,T) = \bD^p(0,T) \times \bH^p(0,T) \times \bL^p_\mu(0,T).$$

Concerning function spaces, in the sequel $\Pi_{pg}(0,T)$ will denote the space of functions $\phi:[0,T] \times \R^d \to \R^k$ of polynomial growth, i.e. for some non negative constants $\delta$ and $C$ 
$$\forall (t,x) \in [0,T] \times \R^d, \quad |\phi(t,x)| \leq C (1+|x|^\delta).$$
For a continuous function $\phi:[0,T] \times \R^d \to \R$ and $\alpha \in [0,1)$, we define
\begin{eqnarray*}
\|\phi \|_\infty &= &\sup_{(t,x)\in [0,T] \times \R^d} |\phi(t,x)|,\\
\|\phi\|_\alpha & = & \sup_{(t,x)\neq (s,y), \ |x-y|\leq1} \frac{|\phi(t,x)-\phi(s,y)|}{|t-s|^{\alpha/2}+|x-y|^\alpha} .
\end{eqnarray*}
For $k\in \N$, $C^{k,2k}=C^{k,2k}([0,T] \times \R^d)$ is the subset of continuous functions $\phi:[0,T] \times \R^d \to \R$ whose partial derivatives of order less than or equal to $k$ w.r.t. $t$ and $2k$ w.r.t. $x$ are continuous on $[0,T] \times \R^d$. For $\alpha \in [0,1)$, the set $H^{k+\alpha/2,2k+\alpha}$ is the subset of $C^{k,2k}$ such that $\|\partial^k_t\phi\|_\alpha + \|\partial^{2k}_x\phi\|_\alpha < +\infty$. We denote $C^k_{l,b}(\R^d)$ the set of $C^k$-functions which grow at most linearly at infinity and whose partial derivatives of order less than or equal to $k$ are bounded. 

\subsection{Our forward backward SDE, assumptions on the coefficients}

First of all we consider the forward SDE: for any $0 \leq t \leq s \leq T$ and any $x\in \R^d$
\begin{equation} \label{eq:SDE}
X^{t,x}_s=x+\int_t^s b(X^{t,x}_r)dr +\int_t^s \sigma(X^{t,x} _r)dW_r + \int_t^s \int_E \beta(X^{t,x}_{r_-},e)\tilde{\mu}(de,dr).
\end{equation}
Moreover for $0\leq s < t$, $X^{t,x}_s = x$. The coefficients $b:\R^d\to \R^d$, $\sigma: \R^d\to \R^{d\times k}$ and $\beta: \R^d\times E \to \R^d$ are supposed to be measurable w.r.t. all variables and satisfy {\bf Conditions (A)}:
\begin{description}
\item[A1.] $b$ and $\sigma$ are Lipschitz continuous w.r.t. $x$, i.e. there exists a constant $K_{b,\sigma}$ such that for any $x$ and $y$ in $\R^d$: 
\begin{equation*} 
|b(x)-b(y)| + |\sigma(x)-\sigma(y)|  \leq K_{b,\sigma} |x-y|
\end{equation*}
\item[A2.] $\beta$ is Lipschitz continuous w.r.t. $x$ uniformly in $e$, i.e. there exists a constant $K_\beta$ such that for all $e \in E$, for any $x$ and $y$ in $\R^d$: 
\begin{equation*}
|\beta(x,e)-\beta(y,e)|\leq K_\beta |x-y|(1\wedge |e|).
\end{equation*}
\item[A3.] There exists a constant $C_\beta$ such that 
\begin{equation*}
|\beta(x,e)| \leq C_\beta (1\wedge |e|).
\end{equation*}
\end{description}
Under these assumptions, for any $(t,x) \in [0,T] \times \R^d$, the forward SDE \eqref{eq:SDE} has a unique strong solution $X^{t,x}=\{X^{t,x}_s, \ t \leq s \leq T\}$. Moreover for all $(t,x) \in [0,T]\times \R^d$ and $p \geq 2$
\begin{equation}\label{eq:integ_X}
\E \left[ \sup_{t\leq s \leq T} |X^{t,x}_s - x |^p \right] \leq C (1+|x|^p)(T-t).
\end{equation}
These results can be found in \cite{prot:04}, chapter V, Theorems 7 and 67. 

The terminal condition $\xi$ of the BSDE will satisfy several assumptions, denoted by {\bf Conditions (B)}. 
\begin{description}
\item[B1.] There exists a function $g$ defined on $\R^d$ with values in $\R_+ \cup \{+\infty\}$ such that 
\begin{equation*}
\xi=g(X^{t,x}_T).
\end{equation*} 
\end{description}
We denote 
$$\mathcal{S}:=\{ x\in \mathbb{R}^d \quad s.t.\quad g(x)=\infty\}$$
the set of singularity points for the terminal condition induced by $g$. This set $\mathcal{S}$ is supposed to be non empty and closed. We also denote by $\bord$ the boundary of $\mathcal{S}$. 
\begin{description}
\item[B2.] Integrability condition: 
\begin{equation*} 
g(X^{t,x}_{T}) \mathbf{1}_{\R^{d} \setminus \cS}(X^{t,x}_{T}) \in \ L^{1} \left( \Omega, \tri_{T}, \Prb  \right).
\end{equation*}
\item[B3.] Continuity condition: $g$ is continuous from $\R^d$ to $\R_+ \cup \{+\infty\}$. 
\end{description}

Now we consider the BSDE: for any $t \leq s \leq T$
\begin{equation}  \label{eq:gene_BSDE}
Y^{t,x}_s = \xi + \int_s^T f(r,X^{t,x}_r,Y^{t,x}_r, Z^{t,x}_r,U^{t,x}_r) dr  -\int_s^T Z^{t,x}_r dW_r- \int_s^T\int_E U^{t,x}_r(e) \tmu(de,dr).
\end{equation}
The generator $f$ of the BSDE \eqref{eq:gene_BSDE} is a deterministic function $f: [0,T] \times \R^d \times \R \times \R^{k} \times \bL^2_\lambda \to \R$. The unknowns are $(Y^{t,x},Z^{t,x},U^{t,x})$. The BSDE is called {\it singular} since the probability $\Prb(\xi = +\infty)$ can be positive.

The function $f$ has the special structure for $u$ in $\bL^2_\lambda$:
\begin{description}
\item[C1.] There exists a function $\gamma$ from $\R^d \times E$ to $\R$ such that 
\begin{equation*}
f(t,x,y,z,u) = f\left(t,x,y,z, \int_E u(e) \gamma(x,e) \lambda(de) \right).
\end{equation*}
\end{description}
For simplicity we denote with the same function $f$ the right and the left hand side. For notational convenience we will denote $f^0_r = f^{0,t,x}_r = f(r,X^{t,x}_r,0,0,0)$.  
\begin{description}
\item[C2.] The process $f^{0,t,x}$ is non negative for any $(t,x) \in [0,T]\times \R^d$. 
\item[C3.] The function $y\mapsto f(t,x,y,z,u)$ has a monotonicity property as follows: there exists $\chi \in \R$ such that for any $t \in [0,T]$, $x \in \R^d$, $z \in \R^k$ and $u \in \R$
\begin{equation*}
(f(t,x,y,z,u)-f(t,x,y',z,u))(y-y') \leq \chi (y-y')^2.
\end{equation*}
\item[C4.] $f$ is locally Lipschitz continuous w.r.t. $y$: for all $R > 0$, there exists $L_R$ such that for any $y$ and $y'$ and any $(t,x,z,u)$
\begin{equation*}
|y|\leq R, |y'|\leq R \Longrightarrow |f(t,x,y,z,u)-f(t,x,y',z,u)| \leq L_R |y-y'|.
\end{equation*}
\item[C5.] $f$ is Lipschitz in $z$, uniformly w.r.t. all parameters: there exists $L > 0$ such that for any $(t,x,y,u)$, $z$ and $z'$:
\begin{equation*}
|f(t,x,y,z,u)-f(t,x,y,z',u)| \leq L |z-z'|.
\end{equation*}
\item[C6.] The function $u \in \R \mapsto f(t,x,y,z,u)$ is Lipschitz and non decreasing for all $(t,x,y,z) \in [0,T] \times \R^d \times \R \times \R^{k}$: 
\begin{equation*} 
\forall u \leq u', \quad  0\leq f(t,x,y,z,u') - f(t,x,y,z,u)\leq L (u'-u).
\end{equation*}
\item[C7.] There exists a function $\vartheta \in \bL^2_\la$ such that 
for all $(x,e) \in \R^d \times E$
\begin{equation*} 
0\leq \gamma(x,e)\leq \vartheta(e).
\end{equation*}
\end{description}
Since the terminal condition may be singular, to ensure that the solution component $Y$ attains the value $\infty$ on $\cS$ at time $T$ but is finite a.s. before time $T$, we suppose that 
\begin{description}
\item[C8.] There exists a constant $q > 0$ and a positive measurable function $a : [0,T]\times \R^d \to \R$ such that for any $y \geq 0$
\begin{equation*}
f(r,X^{t,x}_r,y,z,u)\leq - a(r,X^{t,x}_r) y^{q+1} + f(r,X^{t,x}_r,0,z,u).
\end{equation*}
\end{description}
Moreover, in order to derive the a priori estimate, the following assumptions will hold.
\begin{description}
\item[C9.] The function
$$(t,x)\mapsto \frac{1}{a(t,x)^{1/q}} + f(t,x,0,0,0)$$
belongs to $\Pi_{pg}(0,T)$.  
\item[C10.] There exists $\ell > 1$ such that the function $\vartheta$ in {\bf C5} belongs to $\bL^{\tilde \ell}_\la$ with $\tilde \ell =  \ell/(\ell-1)$.
\end{description}
Again to lighten the notations, $a(r,X_r^{t,x})$ will be denoted $a_r$ or $a^{t,x}_r$ if we do not need to precise the variables $t$ and $x$.

Since we want to work on the link with IPDE, in order to use the work \cite{barl:buck:pard:97}, we need extra assumptions on the regularity of $f$ w.r.t. $t$ and $x$.
\begin{description}
\item[C11.]  The function $t \mapsto f(t,x,y,z,u)$ is continuous on $[0,T]$. 
\item[C12.] For all $R > 0$, $t\in[0,T]$, $|x| \leq R$, $|x'|\leq R$, $|y|\leq R$, $z\in \R^k$, $u \in  \R$,
\begin{equation*}
|f(t,x,y,z,u)-f(t,x',y,z,u)| \leq \varpi_R(|x-x'|(1+|z|)),
\end{equation*}
where $\varpi_R(s)\to 0$ when $s \searrow 0$.
\item[C13.] There exists $C_\gamma > 0$ such that for all $(x,x')\in (\R^d)^2$, $e\in E$,
\begin{equation*}
|\gamma(x,e)-\gamma(x',e)| \leq C_\gamma |x-x'|(1\wedge |e|^2).
\end{equation*}
\end{description}

\begin{defin}[Conditions (C)]
If $f$ satisfies all conditions {\bf C1} to {\bf C13}, we say that $f$ verifies Conditions {\bf (C)}.
\end{defin}

\subsection{Comments on the hypotheses {\bf (C)} and examples}

The previous list is rather long. It is the union of the conditions of \cite{barl:buck:pard:97} and \cite{krus:popi:15}. Let us clarify several points. The condition {\bf C1} is classical (see \cite{barl:buck:pard:97}, \cite{delo:13}, \cite{hama:zhao:14}, etc.)

The conditions {\bf C2} to {\bf C7} are assumed in \cite{krus:popi:14} to ensure that if $\xi$ and $f^0_r$ are in $L^p$ for some $p>1$, the BSDE \eqref{eq:gene_BSDE} has a unique solution in $\bS^p(0,T)$. Indeed by {\bf C4}, for every $n> 0$ the function
\begin{equation*}
\sup_{|y|\leq n} |f(r,X^{t,x}_r,y,0,0)-f^0_r| \leq nL_n 
\end{equation*}
is bounded on $[0,T]$ and thus in $L^1(0,T)$. 
\begin{lem}\label{lem:comp_gene}
Under Hypotheses {\bf C6} and {\bf C7}, for all $(t,x,y,z,u,v) \in [0,T] \times \R^{d+1+k} \times (\bL^2_\lambda)^2$, there exists a progressively measurable process $\kappa = \kappa^{t,x,y,z,u,v} : \Omega \times \R_+ \times E \to \R$ such that
\begin{equation}\label{eq:f_jump_comp} 
f(r,X^{t,x}_r,y,z,u)-f(r,X^{t,x}_r,y,z,v) \leq \int_E (u(e)-v(e))  \kappa^{t,x,y,z,u,v}_r(e)  \la(de)
\end{equation}
with $\Prb \otimes Leb \otimes \la$-a.e. for any $(t,x,y,z,u,v)$, $0 \leq \kappa^{t,x,y,z,u,v}_t(e)$
and $|\kappa^{t,x,y,z,u,v}_t(e)|\leq \vartheta(e)$.
\end{lem}
\begin{proof}
From Hypotheses {\bf C1} and {\bf C6}, we have
\begin{eqnarray*}
&& f(r,X^{t,x}_r,y,z,u)-f(r,X^{t,x}_r,y,z,v) \\
&&\quad = f\left(r,X^{t,x}_r,y,z, \int_E u(e) \gamma(x,e) \lambda(de) \right) - f\left(r,X^{t,x}_r,y,z, \int_E v(e) \gamma(x,e) \lambda(de) \right) \\
&&\quad =  \int_E (u(e)-v(e)) F^{t,x,y,z,u,v}_r\gamma(x,e) \lambda(de) \\
&& \quad =  \int_E (u(e)-v(e)) \kappa_r^{t,x,y,z,u,v}(e) \lambda (de),
\end{eqnarray*}
with 
$$F^{t,x,y,z,u,u'}_r = \frac{ f\left(r,X^{t,x}_r,y,z, \int_E u(e) \gamma(x,e) \lambda(de) \right) - f\left(r,X^{t,x}_r,y,z, \int_E v(e) \gamma(x,e) \lambda(de) \right)}{\int_E (u(e)-v(e)) \gamma(x,e) \lambda(de)}$$
and 
$$\kappa_t^{x,y,z,u,v}(e) = F^{t,x,y,z,u,v}_t\gamma(x,e).$$
Since $f$ is non decreasing and from {\bf C7}, $\kappa_r^{t,x,y,z,u,v}(e) \geq 0$ and from the Lipschitz condition, 
$$|\kappa_r^{t,x,y,z,u,v}(e)| \leq L \vartheta(e).$$
This achieves the proof.
\end{proof}

The previous lemma implies that $f$ is Lipschitz continuous w.r.t. $u$ uniformly in $\omega$, $t$, $y$ and $z$: 
$$|f(t,x,y,z,u)-f(t,x,y,z,v)| \leq L \|\vartheta\|_{L^2_\la} \|u-v\|_{L^2_\la}.$$
Hence we can apply Theorems 1 and 2, together with Proposition 2 in \cite{krus:popi:14} and deduce the existence and the uniqueness of the solution of the BSDE under suitable integrability conditions on $\xi$ and $f^0$. Moreover we can compare two solutions of the BSDE \eqref{eq:gene_BSDE} with different terminal conditions (see Theorem 4.1 and Assumption 4.1 in \cite{quen:sule:13} or Proposition 4 in \cite{krus:popi:14}). 

The assumptions {\bf C8}, {\bf C9} and {\bf C10} are used to deal with singular terminal condition $\xi$, that is when $\Prb(\xi=+\infty)>0$.
\begin{rem} 
Assumption {\bf C8} implies that the function $a$ must be bounded. 
\end{rem}
\begin{lem}\label{lem:a_priori_estimate_cond}
For any $\eta >0$ and $\ell > 0$
\begin{equation}\label{eq:alpha_gamma_2}
\E \int_0^T (T-s)^{-1+\eta} \left[ \left(\frac{1}{qa_s}\right)^{1/q} + (T-s)^{1+1/q} f^0_s \right]^{\ell}ds < +\infty.
\end{equation}
\end{lem}
\begin{proof}
From integrability property \eqref{eq:integ_X} of $X$, {\bf C9} implies that $a^{-1/q}$ and $f^0$ belong to any $L^\delta((0,T)\times \Omega)$ for any $\delta > 1$. Hence:
\begin{eqnarray*}
&& \E \int_0^T (T-s)^{-1+\eta} \left[ \left(\frac{1}{qa^{t,x}_s}\right)^{1/q} + (T-s)^{1+1/q} f^{0,t,x}_s \right]^{\ell}ds\\
&& \quad  \leq C (1+|x|^{\delta \ell})  \E \int_0^T (T-s)^{-1+\eta} ds < +\infty
\end{eqnarray*}
for any $0 < \eta$. 
\end{proof}

This lemma implies in particular that there exists $\ell > 1$ such that
\begin{equation*} 
\E \int_0^T \left[ \left( \frac{1}{qa(r,X_r^{t,x})}\right)^{1/q}+ (T-r)^{1+1/q} f^{0,t,x}_r\right]^{\ell} dr  < +\infty
\end{equation*}
and from {\bf C10}, $\vartheta$ is in $\bL_\lambda^{\tilde \ell}$, where $\tilde \ell$ is the H\"older conjugate of $\ell$.

Finally with Condition {\bf C11}, we will deduce existence of a viscosity solution for the IPDE as in \cite{barl:buck:pard:97}, whereas {\bf C12} and {\bf C13} are assumed to ensure uniqueness of the viscosity solution. 

Now let us give two examples of generators $f$ satisfying Conditions {\bf (C)}. 
\begin{itemize}
\item Assume that $y\mapsto f(y)$ is a non increasing function of class $C^1$, with $f(0)\geq 0$ and such that for some constant $a > 0$ and any $y \geq 0$: $f(y) - f(0) \leq -a y|y|^q$.  Then {\bf (C)} holds. In particular, $f(y)=-y|y|^q$ (for some $q >0$) is a classical example.
\item In \cite{krus:popi:15} the generator related to the optimal closure portfolio strategy is given by:
\begin{eqnarray*}
&& f(t,x,y,u) = -\frac{y|y|^q}{q\eta(t,x)^{q}} + f^0(t,x). 
\end{eqnarray*}
The parameter $\eta > 0$ is the price impact parameter and $f^0 \geq 0$ is the risk measure of the open position. Here $a(t,x)= -\frac{1}{q\eta(t,x)^{q}}$ and $\eta$ and $f^0$ are continuous functions of polynomial growth.  
\end{itemize}

\subsection{A first link with viscosity solution of a IPDE}

First assume that $(t,x) \in [0,T] \times \R^d$ is fixed. Under Conditions {\bf C1} to {\bf C7}, from Theorems 1 or 2 in \cite{krus:popi:14}, there exists a unique solution for the truncated version of BSDE \eqref{eq:gene_BSDE}, where the terminal condition $\xi = g(X^{t,x}_T)$ is replaced by $\xi \wedge n = g(X^{t,x}_T)\wedge n = g_n(X^{t,x}_T)$ and where the generator $f$ is replaced by $f_n$ for some $n>0$:
$$f_n(r,y,z,u) = (f(r,X^{t,x}_r,y,z,u)-f^0_r) + (f^0_r \wedge n).$$
From {\bf B3}, for any $n \in \N^*$, $x \mapsto g_n(x) = g(x) \wedge n$ is a continuous function on $\R^d$. The solution of this truncated BSDE will be denoted by $(Y^{n,t,x},Z^{n,t,x},U^{n,t,x})$: for any $t \leq s \leq T$
\begin{eqnarray}\label{eq:trunc_BSDE}
Y^{n,t,x}_s& = & \xi\wedge n + \int_s^T f_n(r,Y^{n,t,x}_r,Z^{n,t,x}_r,U^{n,t,x}_r) dr -\int_s^T Z^{n,t,x}_r dW_r \\ \nonumber
& -& \int_s^T\int_E U^{n,t,x}_r(e)\tilde{\mu}(dr,de).
\end{eqnarray}
Moreover $(Y^{n,t,x},Z^{n,t,x},U^{n,t,x}) \in \bS^\delta(0,T)$ for any $\delta > 1$. 

If {\bf (C)} holds, we work with almost the same setting as in \cite{barl:buck:pard:97}. The only difference is that $f$ is not Lipschitz continuous w.r.t. $y$. But for a fixed $n$, a straightforward consequence of the comparison principle for BSDE implies that $Y^{n,t,x}$ is bounded by $n(T+1)$ (see Proposition 4 in \cite{krus:popi:14}). We can replace in the BSDE \eqref{eq:trunc_BSDE} our generator $f_n$ by $\widehat f_n$ with $\widehat f_n(t,x,y,z,q) = f_n(t,x,\mathcal{T}_n(y),z,q)$ with $\mathcal{T}_n(y)=(n(T+1)y)/(|y|\vee n(T+1))$. From Condition {\bf C4}, $\widehat f_n$ is Lipschitz w.r.t. $y$.

We will use the notion of viscosity solution of the IPDE \eqref{eq:IPDE}. The reason will be clearer later. For a locally bounded function $v$ in $[0,T]\times \R^d$, we define its upper (resp. lower) semicontinuous envelope $v^*$ (resp. $v_*$) by:
$$v^*(t,x) =\limsup_{(s,y)\to (t,x)} v(s,y) \quad (\mbox{resp. } v_*(t,x) =\liminf_{(s,y)\to (t,x)} v(s,y)).$$
For such equation \eqref{eq:IPDE} we introduce the notion of viscosity solution as in \cite{alva:tour:96} (see also Definition 3.1 in \cite{barl:buck:pard:97} or Definitions 1 and 2 in \cite{barl:imbe:08}). Since we do not assume the continuity of the involved function $u$, we adapt the definition of discontinuous viscosity solution (see Definition 4.1 and 5.1 in \cite{hama:zhao:14}). 
\begin{defin} \label{def:global_visc_sol}
A locally bounded function $v$ is 
\begin{enumerate}
\item \textbf{a viscosity subsolution} of \eqref{eq:IPDE} if it is upper semicontinuous (usc) on $[0,T)\times \R^d$ and if for any $\phi \in C^2([0,T] \times \R^d)$ wherever $(t,x) \in [0,T)\times \R^d$ is a global maximum point of $v-\phi$,
\begin{equation*}
-\frac{\partial }{\partial t} \phi(t,x) - \ope \phi(t,x) - \mathcal{I}(t,x,\phi) -f(t,x,v,(\nabla \phi)\sigma,\mathcal{B}(t,x,\phi)) \leq 0.
\end{equation*}
\item \textbf{a viscosity supersolution} of \eqref{eq:IPDE} if it is lower semicontinuous (lsc) on $[0,T)\times \R^d$ and if for any $\phi \in C^2([0,T] \times \R^d)$ wherever $(t,x) \in [0,T)\times \R^d$ is a global minimum point of $v-\phi$,
\begin{equation*}
-\frac{\partial }{\partial t} \phi(t,x) - \ope \phi(t,x) -  \mathcal{I}(t,x,\phi) -f(t,x,v,(\nabla \phi)\sigma,\mathcal{B}(t,x,\phi))  \geq 0.
\end{equation*}
\item \textbf{a viscosity solution} of \eqref{eq:IPDE} if its upper envelope $v^*$ is  a subsolution and if its lower envelope $v_*$ is a supersolution of \eqref{eq:IPDE}.
\end{enumerate}
\end{defin}
This definition is equivalent to Definition 4.1 in \cite{hama:zhao:14}. We can also give another definition like Definition 5.1 in \cite{hama:zhao:14}. For any $\delta > 0$, the operators $\mathcal{I}$ and $\mathcal{B}$ will be split in two parts:
\begin{eqnarray*}
\mathcal{I}^{1,\delta} (t,x,\phi) & = & \int_{|e|\leq \delta} [\phi(t,x+\beta(x,e))-\phi(t,x) - (\nabla \phi)(t,x)\beta(x,e)] \lambda(de)\\
\mathcal{I}^{2,\delta} (t,x,p,\phi) & = & \int_{|e|> \delta} [\phi(t,x+\beta(x,e))-\phi(t,x) - p \beta(x,e)] \lambda(de),\\
\mathcal{B}^{\delta} (t,x,\phi,v) & = &  \int_{|e|\leq \delta} [\phi(t,x+\beta(x,e))-\phi(t,x)] \gamma(x,e)\lambda(de)\\
&& +  \int_{|e|> \delta} [v(t,x+\beta(x,e))-v(t,x)] \gamma(x,e) \lambda(de).\\
\end{eqnarray*}
\begin{defin} \label{def:local_visc_sol}
A locally bounded and upper (resp. lower) semicontinuous function $v$ is \textbf{a viscosity sub (resp. super) solution} of \eqref{eq:IPDE} if for any $\delta > 0$, for any $\phi \in C^2([0,T] \times \R^d)$ wherever $(t,x) \in [0,T)\times \R^d$ is a global maximum (resp. minimum) point of $v-\phi$ on $[0,T]\times B(x,R_\delta)$,
\begin{eqnarray*}
&& -\frac{\partial }{\partial t} \phi(t,x) - \ope \phi(t,x) - \mathcal{I}^{1,\delta}(t,x,\phi) - \mathcal{I}^{2,\delta}(t,x,\nabla \phi,v) \\
&& \qquad \qquad- f(t,x,v,(\nabla \phi)\sigma,\mathcal{B}^\delta(t,x,\phi,v))\leq 0 \ (\mbox{resp. } \geq 0).
\end{eqnarray*}
\end{defin}
We refer to Remark 3.2 and Lemma 3.3 in \cite{barl:buck:pard:97}, to condition (NLT), Proposition 1 and Section 2.2 in \cite{barl:imbe:08} and to Appendix in \cite{hama:zhao:14} for the discussion (and the proof) on the equivalence between Definitions \ref{def:global_visc_sol} and \ref{def:local_visc_sol}. 

In these two definitions the terminal condition $u(T,.)=g$ is not implied. For the Cauchy problem \eqref{eq:IPDE} with $u(T,.)=g$ where $g$ is a bounded\footnote{This condition can be relaxed. See for example Condition 3.3 in \cite{barl:buck:pard:97}.} and continuous solution, we say that a {\bf comparison principle} holds if: for two functions $u$ and $v$, 
\begin{itemize}
\item $u$ is locally bounded and lsc (resp. $v$ is locally bounded and usc) on $[0,T] \times \R^d$ ;
\item $u$ is a subsolution (resp. $v$ is a supersolution) of \eqref{eq:IPDE} ;
\item $u(T,x) \leq g(x)$ (resp. $v(T,x) \geq g(x)$) ;
\end{itemize}
then $u\leq v$ on $[0,T] \times \R^d$. A comparison principle has two immediate consequences. First if $v$ is a viscosity solution of \eqref{eq:IPDE} such that $v^*(T,.) \leq g \leq v_*(T,.)$ on $\R^d$, then $v$ is a continuous function. Second uniqueness of a continuous and bounded viscosity solution holds.

\vspace{0.5cm}
From Theorem 3.4 and Theorem 3.5 in \cite{barl:buck:pard:97}, we have directly the next result. 
\begin{prop}
Under conditions {\bf (A)} on the coefficients of the SDE \eqref{eq:SDE} and assumptions {\bf (B)} and {\bf (C)} on the terminal condition and on the generator of the BSDE \eqref{eq:gene_BSDE}, the function $u_n(t,x) := Y^{n,t,x}_t$, $(t,x) \in [0,T]\times \R^d$, is the unique bounded (by $n(T+1)$) continuous viscosity solution of \eqref{eq:IPDE} with generator $f_n$ and with terminal condition $u_n(T,.)=g_n$.
\end{prop}

\subsection{Known results on singular BSDE}

In \cite{krus:popi:15}, we extend the result of \cite{popi:06} and \cite{anki:jean:krus:13} concerning BSDE with a singular terminal condition, i.e. when $\Prb(\xi=+\infty) > 0$. Note that the special structure {\bf C1} of the generator is useless here.

\begin{prop}[\cite{krus:popi:15}, Theorem 1] 
Under Conditions {\bf C2} to {\bf C10}, the sequence of processes $(Y^{n,t,x},Z^{n,t,x},U^{n,t,x})$ converges to $(Y^{t,x},Z^{t,x},U^{t,x})$ on $\bS^\ell(t,r)$ for any $t \leq r < T$ and
\begin{itemize}
\item $Y^{t,x}_r \geq 0$ a.s. for any $t \leq r \leq T$.
\item $(Y^{t,x},Z^{t,x},U^{t,x})$ belongs to $\bS^\ell(t,r)$ for any $ t \leq r < T$. 
\item For all $t \leq s \leq s' < T$:
\end{itemize}
\begin{equation*}
Y^{t,x}_s=Y^{t,x}_{s'} + \int_s^{s'} f(r,X^{t,x}_r,Y^{t,x}_r,Z^{t,x}_r,U^{t,x}_r) dr-\int_s^{s'} Z^{t,x}_r dW_r-\int_s^{s'} \int_E U^{t,x}_r(e) \tilde{\mu}(dr,de).
\end{equation*}
\begin{itemize}
\item $(Y^{t,x},Z^{t,x},U^{t,x})$ is a super-solution in the sense that: a.s. 
\end{itemize}
\begin{equation} \label{eq:term_cond_super_sol}
\liminf_{r\to T} Y^{t,x}_r \geq \xi = g(X^{t,x}_T).
\end{equation}
\end{prop}
\begin{defin}
Any process $(\tilde Y, \tilde Z, \tilde U)$ satisfying the previous four items is called \textbf{super-solution} of the BSDE \eqref{eq:gene_BSDE} with singular terminal condition $\xi$. 
\end{defin}
In \cite{krus:popi:15}, we have also proved minimality of the constructed solution in the sense that if $(\widetilde Y, \widetilde Z, \widetilde U)$ is another non negative super-solution, then for all $r \in [t,T]$, $\Prb$-a.s. $\widetilde Y_r \geq Y^{t,x}_r$. 

Let us precise immediately that from the second and/or third items $Y$ is \cad\footnote{French acronym for right-continuous with left limit.} on $[0,T[$. This problem is studied in \cite{popi:16}. But in general we do not know if $Y$ has a left limit at time $T$.

A key point of the proof is the a priori estimate: for any $n$, a.s. for any $t \leq s \leq T$
\begin{equation}\label{eq:a_priori_estimate}
Y^{n,t,x}_s \leq Y^{t,x}_s \leq \frac{K_{\ell,L,\vartheta}}{(T-s)^{1+1/q}} \left\{\E \left( \ \int_s^{T} \left[ \left(\frac{1}{qa_r}\right)^{1/q} + (T-r)^{1+1/q} f^0_r \right]^{\ell} dr \bigg| \tri_s\right) \right\}^{1/\ell}
\end{equation}
where $K_{\ell,L,\vartheta}$ is a non negative constant depending only on $\ell$, $L$ and $\vartheta$ and this constant is a non decreasing function of $L$ and $\vartheta$ and a non increasing function of $\ell$.

\section{Existence of a minimal viscosity solution with singular data}

The minimal solution $Y^{t,x}$ of the singular BSDE \eqref{eq:gene_BSDE} is obtained as the increasing limit of $Y^{n,t,x}$: for any $t \leq s \leq T$
$$\lim_{n\to +\infty}  Y^{n,t,x}_s = Y^{t,x}_s.$$
And it is well known that viscosity solutions are stable by monotone limit. That is the reason why we use this notion of weak solutions. 

We define the function $u$ by:
$$u(t,x) = Y^{t,x}_t.$$
Therefore the sequence $u_n(t,x)$ converges to $u(t,x)$. Since $a$ and $f^0$ depend only on $X^{t,x}$, using Condition {\bf C9} and Property \eqref{eq:integ_X}, the a priori estimate \eqref{eq:a_priori_estimate} becomes: there exist two constants $K>0$ and $\delta >0$ such that for all $(t,x)\in [0,T] \times \R^d$:
\begin{equation} \label{eq:majorationu}
0\leq u_n(t,x) \leq u(t,x) \leq \frac{K}{(T-t)^{1/q}}(1+|x|^{\delta}).
\end{equation}
Since $u_n$ is a continuous function, the function $u$ is lower semi-continuous on $[0,T]\times \R^d$ and satisfies for all $x_0\in \R^d$:
\begin{equation} \label{eq:term_cond_liminf}
\liminf_{(t,x)\to (T,x_0)} u(t,x) \geq g(x_0).
\end{equation}

\subsection{Minimal viscosity solution}

The aim of this section is to prove the following result.
\begin{thm} \label{thm:visc_sol}
Under conditions {\bf (A)-(B)-(C)}, $u(t,x) = Y^{t,x}_t$ is a viscosity solution of the IPDE \eqref{eq:IPDE} on $[0,T[ \times \R^d$. Moreover $u$ is the minimal viscosity solution among all non negative solutions satisfying \eqref{eq:term_cond_liminf}.
\end{thm}

Note that we do not prove the continuity of $u$ because of the lack of uniform convergence of the approximating sequence $u_n$. But we are also not able to show that $u$ is discontinuous.

\begin{proof}
In order to prove that $u$ is a viscosity solution, the main tool is the half-relaxed upper- and lower-limit of the sequence of functions $\left\{ u_{n} \right\}$, i.e.
$$\overline{u}(t,x) = \limsup_{{n \rightarrow + \infty}\atop{(t',x') \rightarrow (t,x)}} u_{n}(t',x') \quad 
\mbox{and} \quad \underline{u}(t,x) = \liminf_{{n \rightarrow + \infty}\atop{(t',x') \rightarrow (t,x)}} u_{n}(t',x').$$
In our case, $u_*=\underline{u} = u \leq  \overline{u} = u^{*}$ because the sequence $\left\{ u_{n} \right\}$ is non decreasing and $u_{n}$ is continuous for all $n \in \N^{*}$. From the estimate \eqref{eq:majorationu}, for all $\eps > 0$, there exists a constant $K_\eps$ such that for every $n \in \N^{*}$ and all $(t,x) \in [0,T-\eps] \times \R^{d}$, 
\begin{equation} \label{eq:majorationu_2}
0\leq u_{n}(t,x) \leq u(t,x) \leq K_\eps(1+|x|^\delta).
\end{equation}
Hence $u^*$ also satisfies \eqref{eq:majorationu} and \eqref{eq:majorationu_2}. In other words $u_n$, $u$ and $u^*$ belong to $\Pi_{pg}(0,T-\eps)$ and the upper bound does not depend on the (singular or not) terminal condition.

Since $u_{n}$ is a viscosity solution of the IPDE \eqref{eq:IPDE}, passing to the limit with a stability result, we can obtain that $u$ (resp. $u^*$) is a supersolution (resp. subsolution) of \eqref{eq:IPDE} on $[0,T[ \times \R^{d}$. The details are left to the reader (see the proof of Theorem 4.1 in \cite{hama:zhao:14} or the results in \cite{alva:tour:96}, \cite{barl:imbe:08} or \cite{bens:saya:02}).

We want to prove minimality of the viscosity solution obtained by approximation among all non negative viscosity solutions. Let us consider a non negative viscosity solution $v$ (in the sense of Definition \ref{def:local_visc_sol}) with terminal condition \eqref{eq:term_cond_liminf}. We prove that for all integer $n$: for all $(t,x) \in [0,T] \times \R^{d}$, $u_{n}(t,x) \leq v_{*}(t,x)$. We deduce that $u \leq v_{*} \leq v$. This result ($u_n \leq v_*$) seems to be a direct consequence of a well-known maximum principle for viscosity solutions (see \cite{barl:94} or \cite{cran:ishi:lion:92} when $\mathcal I=0$, \cite{barl:buck:pard:97}, \cite{barl:imbe:08} or \cite{hama:zhao:14} in general). But to the best of our knowledge, this principle was not proved for solutions which can take the value $+ \infty$ at time $T$, except in \cite{popi:06} in the Brownian setting and for $f(y)= -y|y|^q$. Let us adapt the proof of Proposition 23 in \cite{popi:06} to the IPDE \eqref{eq:IPDE}. The key point is to avoid the terminal time $T$. 

To simplify the notation, we will denote $F$ the following function on $[0,T] \times \R^d \times \R \times \R^d \times \mathbb{S}_d \times \R^2$:
\begin{equation} \label{eq:IPDE_notation}
F(t,x,u,p,X,I,B) = -p b(x) - \frac{1}{2}\tr( X (\sigma \sigma^*)(x)) - I - f(t,x,u,p \sigma(x),B).
\end{equation}
$\mathbb{S}_d$ is the set of symmetric matrices of size $d\times d$.

The beginning of the proof is exactly the same: we fix $\eps > 0$ and $n \geq 1$ and we define $u_{n,\eps} (t,x) = u_{n}(t,x) - \frac{\eps}{t}$. We prove that $u_{n,\eps} \leq v_{*}$ for every $\eps$, hence we deduce $u_{n} \leq v_{*}$. Then we argue by contradiction: we suppose that there exists $(s,z) \in [0,T] \times \R^{d}$ such that $u_{n,\eps}(s,z) - v_{*}(s,z) \geq \nu > 0$. First of all, it is clear that $s$ is not equal to $0$ or $T$, because $u_{n,\eps}(0,z) = -\infty$ and $v_{*}(T,z) \geq g(z)$ from \eqref{eq:term_cond_liminf}. Next the functions $u_{n,\eps}$ and $-v_{*}$ are bounded from above on $[0,T] \times \R^{d}$ respectively by $n(T+1)$ and $0$. Thus, for $(\eta, \varrho) \in (\R^{*})^{2}$, if we define:  
$$m(t,x,y) = u_{n,\eps}(t,x) - v_{*}(t,y) - \frac{\eta}{2} |x-y|^{2} - \varrho \left( |x|^{2} + |y|^{2} \right),$$
$m$ has a supremum $M_{\eta, \varrho}$ on $[0,T] \times \R^{d} \times \R^{d}$ and the penalization terms assure that the supremum is attained at a point $(\hat{t}, \hat{x}, \hat{y}) = (t_{\eta, \varrho}, x_{\eta, \varrho}, y_{\eta, \varrho})$. By classical arguments we prove that if $\varrho$ is sufficiently small 
\begin{equation*} 
|\hat{x}|^{2} + |\hat{y}|^{2} \leq \frac{n(T+1)}{\varrho} \quad \mbox{and} \quad |\hat{x}-\hat{y}|^{2} \leq \frac{2n(T+1)}{\eta}. 
\end{equation*}  
Moreover for $\eta$ large enough, the time $\hat{t}$ satisfies $0 < \hat{t} < T$ (see \cite{popi:06} for the details). 

Now since we avoid the terminal time, for $\eta$ large enough, we can apply Jensen-Ishii's Lemma for non local operator established by Barles and Imbert (Lemma 1 and Corollary 2 in \cite{barl:imbe:08}) with $u_{n,\eps}$ subsolution, $v_{*}$ supersolution and $\phi(x,y) = \frac{\eta}{2} |x-y|^{2} + \chi \left( |x|^{2} + |y|^{2} \right)$ at the point $(\hat{t}, \hat{x}, \hat{y})$. We deduce that for any $\delta > 0$ there exists $\bar \zeta > 0$, $(a,p,X)$, $(b,q,Y)$ such that for any $0 < \zeta < \bar \zeta$
\begin{eqnarray*}
\frac{\eps}{T} + o(\zeta)& \leq & -F(\hat{t},\hat{x},u_{n,\eps}(\hat t,\hat{x}), p, X, I^{n,\eps},B^{n,\eps})  + F(\hat{t},\hat{y},v_{*}(\hat t,\hat{y}),q, Y,I^*,B^*). 
\end{eqnarray*}
where $F$ is the function defined by \eqref{eq:IPDE_notation} and the non local operators are
\begin{eqnarray*}
I^{n,\eps}& = &  \mathcal{I}^{1,\delta}(\hat{t},\hat{x},\phi(.,\hat{y})) +  \mathcal{I}^{2,\delta}(\hat{t},\hat{x},p,u_{n,\eps}(\hat{t},\hat{x}))\\
I^* & = & \mathcal{I}^{1,\delta}(\hat{t},\hat{y},-\phi(\hat{x},.)) +  \mathcal{I}^{2,\delta}(\hat{t},\hat{y},q,v_*(\hat{t},\hat{y})) \\
B^{n,\eps} & = & \mathcal{B}^\delta (\hat{t},\hat{x},\phi(.,\hat{y}),u_{n,\eps}(\hat{t},\hat{x})) \\
B^* & = & \mathcal{B}^\delta (\hat{t},\hat{y},-\phi(\hat{x},.)),v_*(\hat{t},\hat{y})) .
\end{eqnarray*}
Using Conditions {\bf (A)} and {\bf (C)} (in particular {\bf C3}--{\bf C5}--{\bf C6}--{\bf C12}--{\bf C13}) and arguing as in \cite{popi:06}, Proposition 4.1 in \cite{hama:zhao:14} or Theorem 3 in \cite{barl:imbe:08}, one can control the difference of the right-hand side to obtain:
\begin{equation} \label{proofminsol}
\frac{\eps}{T}  + o(\zeta)\leq \omega_1\left( \eta,\varrho,\hat x,\hat y\right) O(\delta) + \omega_2 \left( \eta  |\hat{x}- \hat{y}|^{2}, \varrho (1 + |\hat{x}|^{2}+ |\hat{y}|^{2}), |\hat x-\hat y| \right) 
\end{equation} 
where we have gathered in the $\omega_1$ terms, all terms multiplied by $O(\delta)$. The $\omega_2$ term contains all terms of the form $\eta  |\hat{x}- \hat{y}|^{2}$, $\varrho (1 + |\hat{x}|^{2}+ |\hat{y}|^{2}) $ or $|\hat x-\hat y| $. The details are left to the reader. 

We let $\zeta$ and $\delta$ go to zero and since 
$$\lim_{\eta \to + \infty} \lim_{\varrho \to 0} \left( \frac{\eta}{2} |\hat{x}-\hat{y}|^{2} + \varrho \left( |\hat{x}|^{2} + |\hat{y}|^{2} \right) \right) = 0,$$
the inequality (\ref{proofminsol}) leads to a contradiction taking $\varrho$ sufficiently small and $\eta$ sufficiently large. Hence $u_{n,\eps} \leq v_{*}$ and it is true for every $\eps > 0$, so the result is proved. 
\end{proof}

\subsection{Singular terminal condition}

Now we want to study the behaviour of $u$ at the terminal time $T$. As for the singular BSDE \eqref{eq:gene_BSDE}, the main difficulty is to show that   
$$\limsup_{(t,x) \to (T,x_{0})} u(t,x) \leq g(x_{0}) = u(T,x_{0}).$$
On the set $\cS = \{ g = +\infty\}$, we already have \eqref{eq:term_cond_liminf}. Hence we concentrate ourselves on $\cR=\left\{  g < + \infty \right\}$. We overcome this problem in two steps:
\begin{itemize}
\item We prove that $u^{*}$ is locally bounded on a neighbourhood of $T$ on the open set $\cR$. \item We deduce that $u^{*}$ is a subsolution with relaxed terminal condition and we apply this to demonstrate that $u^{*}(T,x) \leq g(x)$ if $x \in \cR$.
\end{itemize}
To obtain the local boundedness of $u^*$, we add some conditions. We need to control a term due to the covariance between the jumps of the SDE \eqref{eq:SDE} and the jumps of the BSDE \eqref{eq:gene_BSDE}. We make a link between the singularity set $\cS$ and the jumps of the forward process $X$. More precisely we assume \textbf{Conditions (D)}:
\begin{description}
\item[D1.] The boundary $\bord$ is compact and of class $C^2$. 
\item[D2.] For any $x\in\mathcal{S}$, $\lambda$-a.s. 
$$x+\beta(x,e)\in\mathcal{S}.$$
Furthermore there exists a constant $\nu > 0$ such that if $x\in\bord$, then $d(x+\beta(x,e),\bord) \geq \nu$, $\lambda$-a.s. 
\end{description}
These assumptions mean in particular that if $X_{s^-} \in \mathcal{S}$, then $X_s \in \mathcal{S}$ a.s. Moreover if $X_{s^-}$ belongs to the boundary of $\mathcal{S}$, and if there is a jump at time $s$, then $X_s$ is in the interior of $\mathcal{S}$. We prove the next result
\begin{thm} \label{thm:sing_term_visc_sol}
Under conditions {\bf (A)-(B)-(C)-(D)} and if
\begin{equation}\label{eq:cond_def_rho} 
\frac{2}{q} + 2\left( 1 -\frac{1}{\ell} \right) < 1
\end{equation}
holds, then 
$$\lim_{(t,x)\to (T,x_0)} u(t,x) = g(x_0).$$
\end{thm}

\begin{rem}[On the condition \eqref{eq:cond_def_rho}] \label{rem:C+_rho}
The condition \eqref{eq:cond_def_rho} is a balance between the nonlinearity $q$ and the singularity of the generator $f$. It holds for $\ell < 2$ and $q > \frac{2\ell}{2-\ell}$. In other words if $q > 2$, we can take $\ell \in (1,2)$ such that $q > 2\ell/(2-\ell)$. The counterpart is that $\vartheta$ should be in $\bL^2_\lambda \cap \bL^{\tilde \ell}_\lambda$ with $\tilde \ell = \ell/(\ell-1)$. If the generator is $f(y) = - y|y|^q$, it is sufficient to suppose that $q > 2$, which was supposed in \cite{pioz:15} and in \cite{popi:06}.
\end{rem}

We split the proof in two lemmas. 
\begin{lem} \label{lem:bound_on_un}
Under the conditions of Theorem \ref{thm:sing_term_visc_sol}, there exists a constant $C$ independent of $n$ and $t$ such that for $\phi$ defined by \eqref{eq:def_test_func}
$$u_{n}(t,x) \phi(x) \leq C(1+|x|^\delta).$$
\end{lem}
\begin{proof}
To prove the local boundedness of $u^*$, we follow the same scheme as in \cite{popi:16}, Section 4.4. Remember that $\cS$ is the singular set of $g$, $\cR = \cS^c$ is open and for any $\eps > 0$ we define
$$\Gamma(\eps):=\{x\in\cR: \ d(x,\bord) \geq \eps \}.$$
$d(.,\bord)$ is the distance to the boundary $\bord$. By the $C^\infty$ Urysohn lemma, there exists a $C^\infty$ function $\psi$ such that $\psi \in [0,1]$, $\psi\equiv 1$ on $\Gamma(\eps)$ and $\psi\equiv 0$ on $\Gamma(\eps/2)^c$. In particular the support of $\psi$ is included in $\cR$ and since $\bord$ is compact, $\psi$ belongs to $C^\infty_b(\R^d)$. We take $\gamma > 2(q+1)/q$ and we define 
\begin{equation}\label{eq:def_test_func}
\phi = \psi^\gamma.
\end{equation} 
Note that $\phi$ also takes its values in $[0,1]$, $\phi\equiv 1$ on $\Gamma(\eps)$ and $\phi\equiv 0$ on $\Gamma(\eps/2)^c$.

Using Lemma \ref{lem:a_priori_estimate_cond}, since \eqref{eq:cond_def_rho} holds, we can take $\eta > 0$ such that 
\begin{equation}\label{eq:def_rho} 
\rho = \frac{2}{q} + 2\left( 1 -\frac{1}{\ell} \right) + \frac{2\eta}{\ell} < 1 .
\end{equation}
From Proposition 2 in \cite{popi:16}, there exist two constants $C$ and $\delta$ independent of $(n,t,x)$ such that the process $(Z^{n,t,x},U^{n,t,x})$ satisfies:
\begin{equation} \label{eq:sharp_estimate_Z_U}
\E\left[ \int_t^T (T-s)^{\rho} \left( |Z^{n,t,x}_s|^2 + \|U^{n,t,x}_s\|^2_{\bL^2_\lambda} \right) ds  \right]^{\ell/2} \leq C(1+|x|^\delta).
\end{equation}

We use It\^o's formula to the process $Y^{n,t,x}\phi(X^{t,x})$ where $\phi$ is defined by \eqref{eq:def_test_func}, between $t$ and $T$ and we take the expectation since $(Y^n,Z^n,U^n,M^n)$ belongs to $\bS^2(0,T)$, $X$ is in $\bH^2(0,T)$, and $\phi$ and the derivatives of $\phi$ are supposed to be bounded. Thus we obtain for $x \in \R^d$ and $t\in [0,T)$:
\begin{eqnarray}  \label{eq:behaviour_visc_sol_T}
&& u_n(t,x)\phi(x) = \E[Y^{n,t,x}_T\phi(X^{t,x}_T)]  - \E\int_t^T Y^{n,t,x}_{s_-} \left[ \ope \phi(s,X^{t,x}_s)  + \mathcal{I}(s,X^{t,x}_{s^-}, \phi) \right] ds\\ \nonumber
&& \qquad + \E\left[\int_t^T  \phi(X^{t,x}_{s_-}) f_n(s,Y^{n,t,x}_s,Z^{n,t,x}_s,U^{n,t,x}_s) ds\right]\\ \nonumber
&& \qquad - \E\left[\int_t^T \nabla\phi(X^{t,x}_s)\sigma(X^{t,x}_s)Z^{n,t,x}_s ds\right]\\ \nonumber
&&\qquad  - \E\left[\int_t^T \int_E (\phi(X^{t,x}_s)-\phi(X^{t,x}_{s_-}))U^{n,t,x}_s(e)\lambda(de)ds\right].
\end{eqnarray}
From the Assumptions \textbf{B1} and \textbf{B2} on $\xi=g(X^{t,x}_T)$, we have for any $n$:
$$\E (Y^{n,t,x}_T \phi(X^{t,x}_T)) \leq \E (g(X^{t,x}_T)\phi(X^{t,x}_T)) < +\infty.$$
Now we decompose the quantity with the generator $f_n$ as follows:
\begin{eqnarray} \label{eq:decomp_generator}
&&\E\left[\int_t^T \phi(X^{t,x}_{s_-}) f_n(s,X^{t,x}_s,Y^{n,t,x}_s,Z^{n,t,x}_s,U^{n,t,x}_s) ds\right] \\ \nonumber
&& \quad = \E\left[\int_t^T \phi(X^{t,x}_{s_-}) (f(s,X^{t,x}_s,Y^{n,t,x}_s,0,0)-f^{0,t,x}_s) ds\right] \\ \nonumber
&& \qquad +\E\left[\int_t^T \phi(X^{t,x}_{s_-}) (f^{0,t,x}_s\wedge n) ds\right] \\ \nonumber
&&\qquad + \E\left[\int_t^T \phi(X^{t,x}_{s_-}) \zeta^n_s Z^{n,t,x}_s ds\right] + \E\left[\int_t^T \phi(X^{t,x}_{s_-}) \mathcal{U}^{n,t,x}_s ds\right]
\end{eqnarray}
where $\zeta^n_s$ is a $k$-dimensional random vector defined by:
$$\zeta^{i,n}_s = \frac{\left( f(s,X^{t,x}_s,Y^{n,t,x}_s,Z^{n,t,x}_s,0) - f(s,X^{t,x}_s,Y^{n,t,x}_s,0,0)\right) }{Z^{i,n}_s} \ind_{Z^{i,n,t,x}_s\neq 0}$$
and
$$\mathcal{U}^n_s  = f(s,X^{t,x}_s,Y^{n,t,x}_s,Z^{n,t,x}_s,U^{n,t,x}_s) - f(s,X^{t,x}_s,Y^{n,t,x}_s,Z^{n,t,x}_s,0).$$

Now from Conditions {\bf C5} and {\bf C9}, using Property \eqref{eq:integ_X},  Lemma \ref{lem:comp_gene} and Estimate \eqref{eq:sharp_estimate_Z_U}, with $\rho < 1$ given by \eqref{eq:def_rho}, we can prove that there exists a constant $C$ such that for any $n$
\begin{eqnarray} \label{eq:control_cont_1}
&& \E \int_t^T |\left( \nabla\phi(X^{t,x}_s)\sigma(X^{t,x}_s) + \phi(X^{t,x}_{s}) \zeta^n_s\right)Z^{n,t,x}_s| ds \\ \nonumber
&& \quad + \E \int_t^T \left( \int_E |\phi(X^{t,x}_{s})-\phi(X^{t,x}_{s_-})| |U^{n,t,x}_s(e)| \lambda(de) +  | \phi(X^{t,x}_{s_-})| |\mathcal{U}^n_s| \right) ds \\ \nonumber
&& \quad + \E\left[\int_t^T \phi(X^{t,x}_{s_-}) (f^{0,t,x}_s\wedge n) ds\right] \leq C(1+|x|^\delta)
\end{eqnarray}
Moreover since $\gamma > 2(q+1)/q$ in \eqref{eq:def_test_func}, there is a constant $C$ such that 
$$\ope \phi =\ope (\psi^\gamma) \leq C \psi^{\gamma-2},$$
(see \cite{marc:vero:99}). Hence by  H\"older's inequality
\begin{eqnarray} \label{eq:control_cont_3}
&&  \E \left[ \int_t^T |Y^{n,t,x}_{s_-} \ope \phi (s,X^{t,x}_s)| ds \right] \leq C \left[ \E \int_t^T a(s,X^{t,x}_s) \phi(X^{t,x}_s) (Y^{n,t,x}_{s})^{q+1}  ds \right]^{1/(q+1)}.
\end{eqnarray}

Up to now Conditions {\bf (D)} were not used. In fact they are assumed only to control the non local term $\cI$. Since $\bord$ is compact and of class $C^1$, then there exists a constant $\eps_0  >0$ such that for every $y\in \cR \cap \Gamma(\eps_0)^c$, there exists a unique $z\in\bord$ such that $d(y,\bord)=\|y-z\|$ (see for example \cite{gilb:trud:01}, Section 14.6). From \cite{popi:16}, Lemma 4.8, we can choose $\eps_0$ small enough such that for any $0 < \eps < \eps_0$:
\begin{equation}\label{eq:control_jumps_D}
\psi(X_{s^-})=0\Rightarrow\psi(X_s)=0,\qquad \frac{\psi(X_s)}{\psi(X_{s^-})} = \psi(X_s)  \ind_{\Gamma(\eps)}(X_{s^-}).
\end{equation}
These last properties are used to prove as in \cite{popi:16} that:
\begin{equation}\label{eq:control_cont_4}
\E \left[\int_t^T Y^{n,t,x}_{s_-}|\mathcal{I}(s,X^{t,x}_{s^-},\phi)|ds\right]  \leq C \left[ \E \int_t^T a(s,X^{t,x}_s) \phi(X^{t,x}_{s^-})(Y^{n,t,x}_s)^{q+1}ds\right]^{\frac{1}{q+1}} .
\end{equation}

Now we come to the conclusion. By Condition {\bf C8}
\begin{eqnarray}\label{eq:control_gene_power}
&&- \E\left[\int_t^T \phi(X^{t,x}_{s_-}) (f(s,X^{t,x}_s,Y^{n,t,x}_s,0,0) - f^{0,t,x}_s )ds\right]\\ \nonumber
&&\qquad  \geq \E\left[\int_0^t \phi(X^{t,x}_{s_-}) a(s,X^{t,x}_s) (Y^{n,t,x}_s)^{1+q} ds\right] .
\end{eqnarray}
The relations \eqref{eq:control_cont_1}, \eqref{eq:control_cont_3}, \eqref{eq:control_cont_4} and \eqref{eq:control_gene_power} hold.
Thus, we have:
\begin{eqnarray*}
&&-\E \int_{t}^{T} \phi(X^{t,x}_{s}) f_n(s,Y^{n,t,x}_s,Z^{n,t,x}_s,U^{n,t,x}_s) ds \\
&&\qquad + \E\int_{t}^{T} Y^{t,x,n}_{s^-}\left[ \ope \phi(X^{t,x}_{s})  +  \mathcal{I}(s,X^{t,x}_{s^-}, \phi)\right] ds  \leq C(1+|x|^\delta).  
\end{eqnarray*}
The constant $C$ does not depend on $n$ and $t$. In the left hand side, the second term is controlled by the first one raised to a power strictly smaller than 1 (see \eqref{eq:control_cont_3} and \eqref{eq:control_cont_4}). Therefore, there exists a constant $C$:
$$\E \int_{t}^{T}  \phi(X^{t,x}_{s})  |f_n(s,Y^{n,t,x}_s,Z^{n,t,x}_s,U^{n,t,x}_s)|  ds \leq C(1+|x|^\delta).$$
From \eqref{eq:behaviour_visc_sol_T} we deduce that there exists a constant $C$ independent of $n$ and $t$ such that 
$$u_{n}(t,x) \phi(x) \leq C(1+|x|^\delta).$$
This achieves the proof of the lemma.
\end{proof}

\vspace{0.3cm}
From the boundedness of $u^n$ on $[0,T] \times \Gamma(\eps)$, uniformly in $n$, we can derive the next result.
\begin{lem} \label{lem:relaxed_term_cond}
Assumptions {\bf (A)-(B)-(C)-(D)} and \eqref{eq:cond_def_rho} hold. For any $\eps > 0$, if we define the closed subset of $\cR$
$$\Gamma(\eps):=\{x\in\cR: \ d(x,\bord) \geq \eps \}$$
$u^*$ is a subsolution with relaxed terminal condition:
\begin{equation*} 
\left\{ \begin{array}{l}
\displaystyle - \frac{\partial u^{*}}{\partial t} - \ope u^{*} -\mathcal{I} u^* - f(t,x,u^*,\nabla u^* \sigma,\mathcal{B}(t,x,u^*)) = 0, \ \mbox{in} \  [0,T) \times \Gamma(\eps);\\  
\\
\displaystyle \min \left[ - \frac{\partial u^{*}}{\partial t} - \ope u^{*}  -\mathcal{I} u^*- f(t,x,u^*,\nabla u^* \sigma,\mathcal{B}(t,x,u^*)) ; \ u^{*} - g \right] \leq 0, \ \mbox{in} \ \left\{ T \right\} \times\Gamma(\eps).
\end{array} \right.
\end{equation*}
\end{lem}
\begin{proof}
For any $0 < \eps < \eps_0$, $u_{n}$ is bounded on $[0,T] \times \Gamma(\eps)$ by $C(1+|x|^\delta)$ uniformly w.r.t. to $n$. Therefore, $u^{*}$ is bounded on $[0,T] \times \Gamma(\eps)$ by $C(1+|x|^\delta)$. We know that  $u_{n}$ is a subsolution of the IPDE \eqref{eq:IPDE}  restricted to $[0,T] \times \Gamma(\eps)$, i.e. for $ (t,x) \in [0,T[ \times \Gamma(\eps)$
\begin{equation*} 
- \frac{\partial u_{n}}{\partial t}(t,x) - \ope u_{n}(t,x)- \mathcal{I}(t,x,u_n)  - f_n(t,x,u_n,(\nabla u_n)\sigma(t,x),\mathcal{B}(t,x,u_n)) = 0
\end{equation*}
with the terminal condition 
$$u_{n}(T,x) = (g \wedge n)(x), \ x \in \Gamma(\eps).$$
From Theorem \ref{thm:visc_sol}, $u^{*}$ is a subsolution of the IPDE \eqref{eq:IPDE} on $[0,T[\times \Gamma(\eps)$. 

The behaviour at time $T$ is an adaptation of Theorem 4.1 in \cite{barl:94} (see also section 4.4.5 in \cite{barl:94}). Since $g$ is continuous (Hypothesis {\bf B3}), 
$$g(x) = \overline{g}(x) = \limsup_{{n \rightarrow + \infty}\atop{x' \rightarrow x}} (g\wedge n)(x').$$ 
Now assume that for $\vartheta \in C^{1,2}([0,T]\times \R^d) \cap \Pi_{pg}$ such that $u^* -\vartheta$ has a strict global maximum on $[0,T]\times \Gamma(\eps)$ at $(T,x)$ and suppose that $u^*(T,x) > g(x)$. There exists a subsequence $n_k$ such that $(t_{n_k},x_{n_k})$ is the global maximum of $u_{n_k}-\vartheta$ on $[0,T]\times \overline{B(x,R_\delta)}$ and as $k$ goes to $\infty$, $(t_{n_k},x_{n_k})\longrightarrow (T,x)$ and $u_{n_k}(t_{n_k},x_{n_k}) \longrightarrow u^*(T,x)$. This implies in particular that $t_{n_k} < T$ for any $k$ large enough. If not, then up to a subsequence (still denoted $n_k$), 
$$u^*(t,x)=\limsup_{k} u_{n_k}(t_{n_k},x_{n_k}) = \limsup_{k} u_{n_k}(T,x_{n_k})  = \limsup_{k} (g\wedge n_k)(x_{n_k}) \leq g(x) .$$
Since $u_{n_k}$ is a subsolution, we still have 
by Definition \ref{def:local_visc_sol}, 
\begin{eqnarray*} 
&& -\frac{\partial }{\partial t} \vartheta(t_{n_k},x_{n_k}) - \ope \vartheta(t_{n_k},x_{n_k}) - \mathcal{I}^{1,\delta}(t_{n_k},x_{n_k},\vartheta) - \mathcal{I}^{2,\delta}(t_{n_k},x_{n_k},\nabla \vartheta,u_{n_k}) \\ \nonumber 
&&\qquad - f^{n_k}(t_{n_k},x_{n_k},u_{n_k},(\nabla \vartheta)\sigma(t_{n_k},x_{n_k}),\mathcal{B}^\delta(t_{n_k},x_{n_k},\vartheta,u_{n_k})) \leq 0.
\end{eqnarray*}
and passing through the limit we obtain 
\begin{eqnarray*} 
&& -\frac{\partial }{\partial t} \vartheta(T,x) - \ope \vartheta(T,x) - \mathcal{I}^{1,\delta}(T,x,\vartheta) \\
&& \qquad \leq  \mathcal{I}^{2,\delta}(T,x,\nabla \vartheta,\vartheta) + f(T,x,u^{*},(\nabla \vartheta)\sigma(T,x),\mathcal{B}(T,x,\vartheta)) .
\end{eqnarray*}
Thus $u^*$ is a subsolution on $[0,T]\times \Gamma(\eps)$. 
\end{proof}

\vspace{0.5cm}
From Lemma \ref{lem:relaxed_term_cond}, Theorem 4.7 in \cite{barl:94} (with straightforward modifications) shows that $u^{*} \leq g$ in $ \left\{ T \right\} \times \Gamma(\eps)$. In other words for any $x_0 \in \mathcal{R}$, 
$$\limsup_{(t,x)\to (T,x_0)} u(t,x) \leq g(x_0).$$
With Inequality \eqref{eq:term_cond_liminf}, we obtain the desired behaviour of $u$ near terminal time $T$. This achieves the proof of Theorem \ref{thm:sing_term_visc_sol}.

\section{Regularity of the minimal solution} \label{sect:regul}

The function $u$ is the minimal non negative viscosity solution of the PDE \eqref{eq:IPDE}. From \eqref{eq:majorationu} we know that $u$ is finite on $[0,T[ \times \R^{d}$, and for $\eps > 0$ $u$ is bounded on $[0,T-\eps] \times \R^{d}$ by $K(1+|x|^\delta)\eps^{-1/q}$. We cannot expect regularity on $[0,T]\times \R^d$, but only on $[0,T-\eps]\times \R^d$ for any $\eps > 0$. In order to obtain a smoother solution $u$, some assumptions are imposed on the coefficients. We distinguish three different conditions.  
\begin{itemize}
\item {\it Sobolev regularity}. The viscosity solution is a weak solution in the Sobolev sense if the coefficients on the forward SDE \eqref{eq:SDE} are smooth and if the {\it linkage operator} $x\mapsto x +\beta(x,e)$ is a $C^2$-diffeomorphism. 
\item {\it H\"older regularity}. Under some non degeneration assumption on the operators $\ope$ ({\bf A5}) or $\cI$ ({\bf (E)}), then the viscosity solution is locally H\"older continuous. 
\item {\it Strong regularity}. Under the uniform ellipticity condition {\bf A5}, $u$ can be a classical solution under different settings.
\begin{itemize}
\item If the measure $\lambda$ is finite we can transform the IPDE \eqref{eq:IPDE} into some PDE without non local operator (technique developed in \cite{ma:yong:zhao:10} or \cite{pham:98}) and then use regularity arguments for such PDE. 
\item In the setting of \cite{garr:mena:02}, i.e. for some $\gamma < 2$
\begin{equation} \label{eq:gamma_cond_lambda}
\int_E (1\wedge |e|^\gamma) \lambda(de) < +\infty
\end{equation}
and the linkage operator satisfies
\begin{equation} \label{eq:linkage_cond}
\det (\mbox{{\rm Id}}_d + \nabla_x \beta(x,e)) \geq c_1 > 0,
\end{equation}
the existence of a Green function $G$ with suitable properties will ensure a regularizing effect of the operator $\ope + \cI$. 
\end{itemize}
\end{itemize}
Of course, none of these settings gives necessary conditions and other sufficient assumptions could be exhibited.

\subsection{Sobolev regularity of the solution}

The solution $u$ is the increasing limit of $u_n$. For $u_n$ we can apply Theorem 1 of \cite{mato:sabb:zhou:15}. Indeed let us fix a continuous positive and integrable weight function $\rho$ such that $1/\rho$ is locally integrable. We define $\bL^2_\rho([0,T]\times\R^d)$ the Hilbert space of functions $v :[0,T]\times \R^d \to \R$ such that 
$$\int_0^T \int_{\R^d} |v(t,x)|^2 \rho(x) dx dt < +\infty.$$
We assume that 
\begin{itemize}
\item The functions $b$, $\sigma$, $\beta(.,e)$ are in $C^3_{l,b}(\R^d)$ for any $e \in E$. Condition {\bf A3} holds also for all derivatives of $\beta$ of order less than or equal to 3. 
\item For each $e \in E$ the linkage operator $x \mapsto x+ \beta(x,e)$ is a $C^2$-diffeomorphism. 
\end{itemize}
These extra assumptions are used to control the stochastic flow generated by $X^{t,x}$ (see Proposition 2 in \cite{mato:sabb:zhou:15}). 

Recall that we can replace in the BSDE \eqref{eq:trunc_BSDE} our generator $f_n$ by $\widehat f_n$ where $\widehat f_n$ is Lipschitz continuous w.r.t. $y$. Hence all assumptions of Theorem 1 in \cite{mato:sabb:zhou:15} are fulfilled: $u_n(t,x) = Y^{n,t,x}_t$ is the unique Sobolev solution of IPDE \eqref{eq:IPDE} in the space
$$\cH_T = \left\{ v \in L^2_\rho([0,T]\times\R^d) , \quad \sigma^* \nabla v  \in L^2_{\rho}([0,T] \times \R^d)\right\}.$$
The definition of Sobolev solution is given in Definition 1 in \cite{mato:sabb:zhou:15}. Moreover $(\sigma^* \nabla u_n) (t,x) = Z^{n,t,x}_t$. In particular for any $\eps > 0$, and each function $\phi \in C^\infty([0,T]\times \R^d)$ with compact support in $\R^d$, for any $t \leq T-\eps$
\begin{eqnarray*}
&& \int_t^{T-\eps} (u_n(s,.),\partial_s\phi(s,.))ds + (u_n(t,.),\phi(t,.))-(u_n(T-\eps,.),\phi(T-\eps,.)) \\
&& \qquad - \int_t^{T-\eps} (u_n(s,.),\cA^*\phi(s,.))ds \\
&&\quad = \int_t^{T-\eps} (f(s,.,u_n(s,.),\sigma^*\nabla u_n(s,.), \mathcal{B}(s,.,u_n)),\phi(s,.))ds 
\end{eqnarray*}
where $(v,w) = \int_{\R^d} u(x) v(x)dx$ is the scalar product on $L^2(\R^d)$ and $\cA^*$ is the adjoint operator of the operator $\ope+ \cI$. 

Moreover for any $\eps  >0$, on $[0,T-\eps]$, by the estimate \eqref{eq:majorationu}, and from the inequality \eqref{eq:sharp_estimate_Z_U} if \eqref{eq:cond_def_rho} holds, we deduce that $u_n$ and $\sigma^* \nabla u_n$ are bounded from above by $C(1+|x|^\delta)$ for some $C> 0$ and $\delta > 0$. Hence if we choose the suitable weight $\rho$, $u_n$ and $\sigma^* \nabla u_n$ are bounded in $\cH_{T-\eps}$. Therefore the next result is proved. 
\begin{prop} \label{prop:sobolev_reg}
Under conditions {\rm \textbf{(A)}-\textbf{(B)-\textbf{(C)}}} and \eqref{eq:cond_def_rho}, if the coefficients $b$, $\sigma$ and $\beta$ satisfy the above conditions, then $u \in \cH_{T-\eps}$ and is a Sobolev solution of the IPDE \eqref{eq:IPDE} on $[0,T-\eps]$ for any $\eps > 0$. 
\end{prop}

Note that in the case $f(y) = - y|y|^q$, the only hypotheses in order to have a Sobolev solution are on the coefficients of the forward diffusion.

\subsection{Lipschitz/H\"older regularity of the solution} \label{sect:holder_reg} 

Recently there have been several papers \cite{barl:chas:ciom:imbe:12,barl:chas:imbe:11,caff:silv:09,chan:davi:14a,chan:davi:14b,silv:10} (among many others) dealing with $C^\al$ estimates and regularity of the solution of the IPDE \eqref{eq:IPDE}. Here we will mainly use the papers \cite{barl:chas:ciom:imbe:12,barl:chas:imbe:11}. 

In our setting we defined $F$ by \eqref{eq:IPDE_notation} and from Conditions {\bf (A)} and {\bf (C)} we can easily check that $F$ is continuous and {\it degenerate elliptic} and {\it (H0)} and {\it (H2)} of \cite{barl:chas:ciom:imbe:12} hold:
\begin{itemize}
\item If $X \geq Y$, $I \geq I'$, $B\geq B'$, $F(t,x,u,p,X,I,B) \leq F(t,x,u,p,Y,I',B')$.
\item For any $t \in [0,T]$, $x \in \R^d$, $u$, $v$ in $\R$, $p \in \R^d$, $X \in \mathbb{S}_d$ and $(I,B) \in \R^2$, 
$$F(t,x,u,p,X,I,B)-F(t,x,v,p,X,I,B)\geq 0, \ \mbox{when } u\geq v.$$
\item $(I,B)\mapsto F(.,I,B)$ is Lipschitz continuous, uniformly with respect to all the other variables.
\end{itemize}
Since we are just interesting in a global regularity property, we add the {\it strict ellipticity condition} {\it (H)} of \cite{barl:chas:ciom:imbe:12}. In our setting $F$ is linear w.r.t. $X$ and $I$. Hence with $\Lambda_1(x) = 1$ and if $2 \Lambda_2(x)\geq 0$ is the minimal eigenvalue of the matrix $\sigma \sigma^{*}(x)$, we have $\Lambda_1(x) +\Lambda_2(x) \geq 1$ and
$$F(t,x,u,p,Y,I',B) - F(t,x,u,p,X,I,B) \leq  \Lambda_1(x) (I-I') +\Lambda_2(x) \tr (X-Y). $$
Moreover from our conditions {\bf (A)} and {\bf C12} if we assume that $\varpi_R$ in {\bf C12} does not depend on $R$, then there exists a modulus of continuity $\varpi_F$ such that Condition {\it (H)} is satisfied (see Section 4.1 in \cite{barl:chas:ciom:imbe:12}). As explained in the introduction of \cite{barl:chas:ciom:imbe:12}, the diffusion term gives the ellipticity in certain directions whereas it is given by the non local term in the complementary directions. 

Now when the strict ellipticity is involved by the non local terms, we need some extra conditions on the L\'evy measure $\lambda$ and on the coefficient $\beta$ in the SDE \eqref{eq:SDE}. These assumptions are denoted by {\it (J1)} to {\it (J5)} in \cite{barl:chas:ciom:imbe:12}. In the following, $B$ is the unit ball in $E$ and $B_\eps$ is the ball centred at zero with radius $\eps > 0$. Remember that $\lambda$ is a L\'evy measure on $E=\R^d\setminus \{0\}$:
$$\int_E (1\wedge |e|^2) \lambda(de) < +\infty.$$
From Condition {\bf A3}, there exists a constant $C_{\lambda,\beta}$ such that for all $x\in \R^d$:
$$\int_B |\beta(x,e)|^2 \lambda(de) + \int_{\R^d\setminus B} \lambda(de) \leq C_{\lambda,\beta}.$$
From {\bf A2} there exists a constant $K_\beta$ such that for all $e \in B$, and $x$ and $y$ in $\R^d$:
$$|\beta(x,e)-\beta(y,e)| \leq K_\beta |e| |x-y|.$$
Moreover from {\bf A3} for all $(e,x)\in E \times \R^d$ 
$$|\beta(x,e)| \leq C_\beta |e|.$$
Assumption {\bf A2} implies that for all $e \in E$, $|e|\geq 1$, and $x$ and $y$ in $\R^d$:
$$|\beta(x,e)-\beta(y,e)| \leq K_\beta |x-y|.$$
Thereby to verify all conditions of \cite{barl:chas:ciom:imbe:12} we add these additional {\bf hypotheses (E)}.
\begin{description}
\item[E1.] There exists $c_\beta > 0$ such that for all $(e,x)\in E \times \R^d$, $c_\beta |e| \leq |\beta(x,e)|$.
\item[E2.] There exists $\tau \in (0,2)$ such that for every $a \in \R^d$, there exists $0 < \eta < 1$ and a constant $\widetilde C_\lambda>0$ such that the following holds for any $x \in \R^d$
$$\forall \eps > 0, \quad \int_{C_{\eta,\eps}(a)} |\beta(x,e)|^2 \lambda (de) \geq \widetilde C_\lambda \eta^{\frac{d-1}{2}} \eps^{2-\tau}$$
with $C_{\eta,\eps}(a)  = \{ e; \ |\beta(x,e)|\leq \eps, \ (1-\eta)|\beta(x,e)||a| \leq |a.\beta(x,e)| \}$.
\item[E3.] There exists $\tau \in (0,2)$ such that for $\eps > 0$ small enough 
$$\int_{B\setminus B_\eps} |e|\lambda(de) \leq \left\{ \begin{array}{ll} \widehat C_\lambda \eps^{1-\tau} & \mbox{for }  \tau \neq 1,\\ \widehat  C_\lambda |\ln(\eps)| & \mbox{for } \tau =1.\end{array} \right.$$
\end{description}
We denote by {\bf (E)} the three conditions {\bf E1}, {\bf E2}, {\bf E3}. If {\bf (E)} holds then Conditions {\it (J1)} to {\it (J5)} of \cite{barl:chas:ciom:imbe:12} are satisfied. Recall that our terminal condition $g$ is continuous from $\R^{d}$ to $\R\cup\{+\infty\}$ (Hypothesis {\bf B3}). Now we state the main result of this part. This Proposition is a modification of Corollary 7 of \cite{barl:chas:ciom:imbe:12}. 
\begin{prop}
Assume that Conditions {\bf (A)}-{\bf (B)}-{\bf (C)}-{\bf (E)} are satisfied. Moreover the modulus of continuity in {\bf C12} does not depend on $R$. 
\begin{itemize}
\item Assume that $\tau > 1$ and that for all  $M \geq 0$, $g$ is a Lipschitz continuous function on the set $\mathcal{O}_{M} = \left\{ |g| \leq M \right\}$. Then for all $\eps > 0$, $u$ is locally Lipschitz continuous on $\R^d$, uniformly w.r.t. $t \in [0,T-\eps]$: for all $M$, there exists a constant $C_{M,\eps}$ such that 
$$\forall |x|\leq M, \ \forall |y| \leq M, |u(t,x)-u(t,y)|\leq C_{M,\eps} |x-y|.$$
The constant $C_{M,\eps}$ depends only on $\eps$, on $M$ on the dimension d, and on the constants in Assumption {\bf (E)}. 
\item If $\tau \leq 1$, and if for some $\al < \tau$, $g$ is $\al$-H\"older continuous function on the set $\mathcal{O}_{M} = \left\{ |g| \leq M \right\}$ for all  $M \geq 0$, then $u$ is locally $\alpha$-H\"older continuous on $\R^d$, uniformly w.r.t. $t \in [0,T-\eps]$.
\end{itemize}
\end{prop}
\begin{proof}
For any $n \in \N$, $u_n$ is a continuous viscosity solution of \eqref{eq:IPDE} with terminal condition $g_n$. Moreover this function is bounded on $[0,T]\times \R^d$ by $n(T+1)$. We can apply to $u_n$ the results of Theorem 6 and Corollary 7 in \cite{barl:chas:ciom:imbe:12}.
\begin{itemize}
\item Assume that $\tau > 1$. Our condition on $g$ implies that $g_n$ is Lipschitz on $\R^d$. From \cite{barl:chas:ciom:imbe:12}, $u_n$ is locally Lipschitz continuous w.r.t. $x$ on $[0,T]$, i.e. for any $M> 0$, there exists a constant $C_{M,n}$ such that for all $t \in [0,T]$, all $(x,y)\in (\R^d)^2$, $|x|\leq M$ and $|y|\leq M$ 
$$|u_n(t,x)-u_n(t,y)|\leq C_{M,n} |x-y|.$$
The key point here is that the constant $C_{M,n}$ depends only on $M$, on $\|u_n\|_\infty$, on the dimension d, and on the constants in Assumption {\bf (E)}. From the upper bound \eqref{eq:majorationu} (or \eqref{eq:majorationu_2}), we deduce that for any $\eps > 0$, on $[0,T-\eps]$, $\|u_n\|_\infty$ is controlled uniformly w.r.t. $n$. Thus $u_n$ is locally Lipschitz continuous w.r.t. $x$ and the Lipschitz constant $C_{M,n} = C_{M,\eps}$ does not depend on $n$. The pointwise convergence of $u_n$ to $u$ implies that the limit $u$ is locally Lipschitz continuous with the same constants. 
\item If $\tau \leq 1$, then $u_n$ is locally $\alpha$-H\"older continuous w.r.t. $x$ on $[0,T]$. The conclusion follows by the same arguments as above. 
\end{itemize}
\end{proof}

\begin{rem}[On Condition {\bf (E)}]
If the matrix $\sigma \sigma^{*}$ is uniformly elliptic (see Condition {\bf A5}), then the conclusion of the previous proposition still holds without Assumption {\bf (E)} and the regularity of $u$ depends on the regularity of $g$ (no more on $\tau$). Nevertheless Condition {\bf (E)} is crucial to have regularity estimates when the local second order differential operator $\ope$ becomes degenerated. 
\end{rem}

\subsection{Strong regularity of the solution}

Here we explain how we can derive that this minimal viscosity solution is a regular function on $[0,T-\eps]\times \R^d$ under additional conditions. To have more regularity on the solution $u$ we will assume that:
\begin{description}
\item[A4.] $\sigma$ and $b$ are bounded: there exists a constant $C$ s.t.  
\begin{equation*} 
\forall x \in \R^{d}, \quad |b(x)| + |\sigma (x) | \leq C;
\end{equation*}
\item[A5.] $\sigma \sigma^{*}$ is uniformly elliptic, i.e. there exists $\Lambda_0 > 0$ s.t. for all $x \in \R^{d}$:
\begin{equation*} 
\forall y \in \R^{d}, \ \sigma \sigma^{*}(x)y \cdot y \geq \Lambda_0 |y|^{2}.
\end{equation*}
\end{description}

If $\lambda$ is a finite measure, we obtain a regularity result with a transformation of the non local operator $\cI$. If $v$ is a (classical) solution of the IPDE \eqref{eq:IPDE}, that is  
$$\partial_{t} v + \ope v +\mathcal{I}(t,x,v) +f(t,x,v,\nabla v \sigma(t,x), \cB(t,x,v)) = 0,$$
if we define the drift term $\widetilde b$:
$$ \widetilde b(x) = b(x)  - \int_E \beta(x,e) \lambda(de),$$
the differential operator $\widetilde \ope$:
$$\widetilde \ope \phi  =\frac{1}{2}\tr( D^2\phi \sigma \sigma^* (x)) + \widetilde b(x) \nabla\phi,$$
and the generator $\widetilde f$:
\begin{eqnarray*}
&& \widetilde f(t,x,v,\nabla v \sigma,\cB(t,x,v)) = \\
&& \qquad f(t,x,v(t,x),\nabla v(t,x)\sigma(t,x), \cB(t,x,v)) + \int_E [v(x+\beta(x,e))-v(x)] \lambda(de),
\end{eqnarray*}
then 
\begin{equation} \label{eq:cauchyproblem}
\partial_{t} v + \widetilde \ope v + \widetilde f(t,x,v,\nabla v \sigma,\cB(t,x,v)) = 0,
\end{equation}
This transformation allows us to use all results concerning the PDE \eqref{eq:cauchyproblem}, especially the ones contained in the reference book \cite{lady:solo:ural:68}. This idea is used in \cite{ma:yong:zhao:10} and \cite{pham:98}. 
\begin{lem} \label{lem:existregulsol}
Under {\bf (A)}--{\bf A4}--{\bf A5} and {\bf (C)}, if the measure $\lambda$ is finite and if one of the next conditions holds:
\begin{itemize}
\item $f$ depends only on $(t,x,v)$ with $f(t,x,0)$ bounded uniformly w.r.t. $(t,x)$,
\item $(t,x)\mapsto f(t,x,y,z,u)$ is in $H^{\al/2,\al}$, the $(\alpha/2,\alpha)$-H\"older
norm is uniformly bounded w.r.t. $(y,z,u)$ and $f(t,x,0,z,u)$ is bounded, 
\end{itemize}
then for every bounded and continuous function $\phi$, the IPDE \eqref{eq:IPDE} with terminal condition $v(T,.) = \phi$, has a unique bounded classical solution $v$ in the sense that $v \in C^{1,2}([0,T) \times \R^d) \cap C([0,T]\times \R^d)$.
\end{lem}

\begin{proof}
For the first case we use the scheme done in \cite{popi:06} (Proposition 24) on PDE \eqref{eq:cauchyproblem}, for the second the technique developed in Ma et al. \cite{ma:yong:zhao:10} (Theorem 1) and in Pham \cite{pham:98} (Proposition 5.3). The details of the proof are left to the reader.
\end{proof}

\begin{rem}
Let us emphasize that in the first case no regularity assumption on the coefficients $b$, $\sigma$ and $f(.,.,0)$ is required. Only boundedness and {\bf A5} are important. 
\end{rem}

The main drawback of the previous lemmas is the finiteness of the measure $\lambda$. To avoid this condition, we must use regularity results on IPDE. 
\begin{lem} \label{lem:existregulsol_4}
Under {\bf (A)}--{\bf A4}--{\bf A5} and {\bf (C)}, assume that 
\begin{itemize}
\item The measure $\lambda$ satisfies \eqref{eq:gamma_cond_lambda} for some $\gamma < 2$.
\item The function $\beta$ is differentiable w.r.t. $x$ and \eqref{eq:linkage_cond} holds. 
\item $f$ is H\"older-continuous w.r.t. $(t,x)$ uniformly w.r.t. the other parameters, that is there exists $\alpha \in (0,1)$ such that there exists a constant $C$ such that for all $(y,z,u)$
$$|f(t,x,y,z,u)-f(t',x',y,z,u)|\leq C (|t-t'|^{\alpha/2} + |x-x'|^\alpha).$$
\end{itemize}
Then for any bounded and continuous function $\phi$, the IPDE \eqref{eq:IPDE} with terminal condition $\phi$ has a unique solution $v$ in the set $C([0,T]\times \R^d)\cap H^{1+\delta/2,2+\delta}([0,T)\times \R^d)$ where $\delta=\alpha$ if $\gamma < 2-\alpha$ and $\delta \in (0,2-\alpha)$ if $\gamma\in [2-\alpha,2)$. Moreover $v$ is bounded on $[0,T]\times \R^d$.
\end{lem}
\begin{proof}
From \cite{garr:mena:92,garr:mena:93,garr:mena:02}, there exists a unique Green function $G$ associated with the parabolic second-order integro-differential operator $\partial_t u - \ope u - \cI u$.  The key properties of $G$ are inherited from the properties of the Green function $G_\ope$ associated to $\ope$, studied in Chapter IV, Sections 12 to 14 of \cite{lady:solo:ural:68}. From Theorem VIII.2.1 of \cite{garr:mena:92}, Theorem 3.2 of \cite{garr:mena:93} and Section IV.14 of \cite{lady:solo:ural:68}, if $\phi$ is a continuous and bounded function, and $f \in H^{\al/2,\al}$, then the function $v$ defined by
$$v(t,x) = \int_{\R^d} G(x,t,y,0) \phi(y) dy + \int_0^t \int_{\R^d} G(x,t,y,s) f(y,s) dy ds = v_1(t,x)+v_2(t,x)$$
on $(0,+\infty)\times \R^d$ is in $H^{1+\delta/2,2+\delta}((0,\tau]\times \R^d)$ for any $\tau > 0$, and solves the IPDE:
$$\partial_t v - \ope v - \cI v = f(t,x)$$
with the initial condition $v(0,.)=\phi$. Moreover there exists a constant $C$ independent of $\phi$ such that:
$$\eps^{1+\delta/2} \|v_1\|_{H^{1+\delta/2,2+\delta}([\eps,\tau] \times \R^d)} \leq C\|\phi\|_\infty,$$
and
$$\|v_2\|_{H^{1+\delta/2,2+\delta}([\eps,\tau] \times \R^d)} \leq C\| f \|_\delta.$$
The time reversion will give the same inequality on $[0,T-\eps]\times \R^d$, if $\phi$ is a terminal condition at time $T$. 

Using this regularizing property, we can follow the ideas developed in the proof of Theorem 4.2, Chapter VI (see also Theorem 6.1, Chapter V) in \cite{lady:solo:ural:68}. For $0\leq \rho \leq 1$, we consider the family of linear problems: on $[0,T)\times \R^d$
\begin{equation}\label{eq:schauder_IPDE}
\partial_t v(t,x) + \ope v(t,x) + \mathcal{I}(t,x,v) +\rho f(t,x,w,(\nabla w)\sigma,\mathcal{B}(t,x,w)) = 0
\end{equation}
with terminal condition $v(T,.)=\phi$. This defines an operator $\Psi$ which associates each function $w \in H^{1+\delta/2,2+\delta}([0,T)\times \R^d)$ with a solution $v$ of the linear problem \eqref{eq:schauder_IPDE}: $v =\Psi(w,\rho)$. 
The Leray-Schauder principle (see \cite{lady:solo:ural:68} for the details) implies that 
for each $\rho \in [0,1]$, there exists at least one fixed point $v^\rho \in H^{1+\delta/2,2+\delta}([0,T) \times \R^d)$ for $\Psi$. We just have to take $\rho=1$ to obtain $v$.  

The continuity and the boundedness on the solution $v$ comes from classical a priori estimate on the BSDE (see \cite{barl:buck:pard:97}). 
\end{proof}

\begin{prop} \label{prop:regu}
Under {\bf (A)}--{\bf A4}--{\bf A5}, {\bf (B)} and {\bf (C)}, we assume that the conditions of Lemmas \ref{lem:existregulsol} or \ref{lem:existregulsol_4} hold. Moreover the function in {\bf C9} is bounded. Then:
\begin{equation} 
u \in C^{1,2} ([0,T) \times \R^{d}; \R^{+}).
\end{equation}
\end{prop}
Before the proof let us remark that if {\bf (D)} is also satisfied, then $u$ is continuous on $[0,T] \times \R^{d}$. 

\begin{proof}
Recall that $u_{n}$ is jointly continuous in $(t,x)$ and from \eqref{eq:majorationu_2}, $u_{n}$ is bounded on $[0,T-\eps] \times \R^{d}$ uniformly in $n$. But if the function involved in {\bf C9} is bounded, then we can take $\delta=0$ in Lemma \ref{lem:existregulsol_4}. Thus, the problem 
$$\partial_{t} v + \ope v +\mathcal{I}(t,x,v) +f_n(t,x,v,\nabla v \sigma(t,x), \cB(t,x,v)) = 0,$$
with terminal condition $\phi = u_{n}(T-\eps,.)$ has a bounded classical solution. Since every classical solution is a viscosity solution and since $u_{n}$ is the unique bounded and continuous viscosity solution, we deduce that:
\begin{equation*}
\forall \eps > 0, \ u_{n} \in C^{1,2} ([0,T-\eps[ \times \R^{d}; \R^{+}).
\end{equation*}

From the construction of the classical solution $u_n$, we also know that the sequence $\left\{ u_{n} \right\}$ is bounded in $H^{\al,1+\al}([0,T- \eps/2] \times K)$ for any compact subset $K$ of $\R^d$. Among other parameters the bound is given by the $L^{\infty}$ norm of $u_n$ which is smaller than $(T - \eps/4)^{-1/q}$. Therefore $u$ is continuous on $[0,T-\eps/2] \times \R^{d}$ and if we consider the IPDE \eqref{eq:IPDE} with continuous terminal data $u(T-\eps,.)$, with the same argument as for $u_{n}$, we obtain that $u$ is a classical solution, i.e. $u \in C^{1,2} ([0,T-\eps] \times \R^{d}; \R^{+})$. This achieves the proof.
\end{proof}

\vspace{1cm}
\noindent {\bf Acknowledgements.} We would like to thank sincerely the anonymous referee for helpful comments and suggestions. His elaborate report improved substantially this paper.

\bibliography{biblio_sing_BSDE_jumps}

\end{document}